\documentclass[11pt]{amsart}

\usepackage[colorlinks=true, pdfstartview=FitV, linkcolor=blue, citecolor=blue, urlcolor=blue, breaklinks=true]{hyperref}
\usepackage{amsmath,amsfonts,amssymb,amsthm,amscd,comment,paralist,etoolbox,mathtools,enumitem}
\usepackage[usenames,dvipsnames]{color}
\usepackage{mdwlist}

\makeatletter
\def\namedlabel#1#2{\begingroup
    #2%
    \def\@currentlabel{#2}%
    \phantomsection\label{#1}\endgroup
}

\newcommand\C{\mathbb{C}}
\newcommand\Z{\mathbb{Z}}

\newcommand\N{\mathbb{N}}
\newcommand\F{\mathbb{F}}

\newcommand\id{\mathrm{Id}}

\newcommand\cB{\mathcal{B}}

\newcommand\tr{\mathrm{tr}}

\newcommand{\md}{\textup{-mod}}
\newcommand{\rmd}{\textup{-rmod}}
\newcommand{\rR}{\textup{R}}
\newcommand{\rL}{\textup{L}}

\newcommand\pl[1]{\prescript{\ell}{}{#1}}
\newcommand\pr[1]{\prescript{r}{}{#1}}

\newcommand\twBA{\prescript{\beta}{B}{A}_A^{\alpha}}

\newcommand\ts{}

\DeclareMathOperator{\im}{im} 
\DeclareMathOperator{\Hom}{Hom}
\DeclareMathOperator{\HOM}{HOM}

\newtheorem{theo}{Theorem}[section]
\newtheorem{prop}[theo]{Proposition}
\newtheorem{lem}[theo]{Lemma}
\newtheorem{cor}[theo]{Corollary}

\theoremstyle{definition}
\newtheorem{defin}[theo]{Definition}
\newtheorem{rem}[theo]{Remark}
\newtheorem{eg}[theo]{Example}

\numberwithin{equation}{section}
\allowdisplaybreaks

\newtoggle{comments}
\newtoggle{details}
\newtoggle{prelimnote}
\newtoggle{detailsnote}

\toggletrue{detailsnote}   

\iftoggle{comments}{%
  \newcommand{\comments}[1]{
    \begin{center}
      \parbox{6.5 in}{
        \color{red}
          {\footnotesize \textbf{Comments:} #1}
        \color{black}}
    \end{center}}
}{%
  \newcommand{\comments}[1]{}
}

\iftoggle{details}{%
  \newcommand{\details}[1]{
      \ \\
      \color{OliveGreen}
        {\footnotesize \textbf{Details:} #1}
      \color{black}
      \\
  }
}{%
  \newcommand{\details}[1]{}
}

\iftoggle{prelimnote}{%
  \newcommand{\prelim}{\textsc{Preliminary version} \bigskip}
}{%
  \newcommand{\prelim}{}
}

\begin{document}
%

\title{Twisted Frobenius extensions of graded superrings}

\author{Jeffrey Pike}
\address{J.~Pike: Department of Mathematics and Statistics, University of Ottawa}
\email{jpike061@uottawa.ca}

\author{Alistair Savage}
\address{A.~Savage: Department of Mathematics and Statistics, University of Ottawa}
\urladdr{\url{http://AlistairSavage.ca}}
\email{alistair.savage@uottawa.ca}
\thanks{The second author was supported by a Discovery Grant from the Natural Sciences and Engineering Research Council of Canada.  The first author was supported by the Discovery Grant of the second author.}

\begin{abstract}
  We define twisted Frobenius extensions of graded superrings.  We develop equivalent definitions in terms of bimodule isomorphisms, trace maps, bilinear forms, and dual sets of generators.  The motivation for our study comes from categorification, where one is often interested in the adjointness properties of induction and restriction functors.  We show that $A$ is a twisted Frobenius extension of $B$ if and only if induction of $B$-modules to $A$-modules is twisted shifted right adjoint to restriction of $A$-modules to $B$-modules.  A large (non-exhaustive) class of examples is given by the fact that any time $A$ is a Frobenius graded superalgebra, $B$ is a graded subalgebra that is also a Frobenius graded superalgebra, and $A$ is projective as a left $B$-module, then $A$ is a twisted Frobenius extension of $B$.
\end{abstract}


\prelim

\maketitle
\thispagestyle{empty}

\tableofcontents

%
\section{Introduction}
%

Frobenius algebras and their natural generalization, Frobenius extensions, have played an increasingly important role in many areas of mathematics.  For example, commutative Frobenius algebras are the same as two-dimensional topological quantum field theories.  More recently, Frobenius algebras and Frobenius extensions have become ubiquitous in the field of categorification.  Many constructions in categorical representation theory involve categories of modules over Frobenius (graded super)algebras.  There it is important that one has nice adjointness properties of the functors of induction and restriction between categories of modules over a ring and modules over a subring.  In particular, if $B$ is a subring of $A$, then induction from the category of $B$-modules to the category of $A$-modules is always left adjoint to restriction from the category of $A$-modules to the category of $B$-modules.  It is also right adjoint \emph{precisely} when $A$ is a Frobenius extension of $B$ (see \cite[Th.~5.1]{Mor65}).  However, there are many important situations when something weaker happens.  For example, if $A$ is a Frobenius algebra, $B$ is a subalgebra of $A$ that is also a Frobenius algebra, and $A$ is projective as a left $B$-module, then induction is \emph{twisted} right adjoint to restriction.  This follows, for example, from \cite[Lem.~1]{Kho01}, or see \cite[Prop.~6.7]{RS15} for the statement in the graded super setting of the current paper.  However, it turns out that the converse does not hold.  It is thus natural to ask what conditions on $A$ and $B$ correspond precisely to induction being twisted right adjoint to restriction.  (See Theorem~\ref{theo:twisted-adjointness} for the type of twisted adjointness we consider in the current paper.)  It follows from work of Morita (\cite{Mor65}) that the correct notion is that of an $(\alpha,\beta)$-Frobenius extension.

The goal of the current paper is to develop the theory of $(\alpha,\beta)$-Frobenius extensions (which we also call \emph{twisted} Frobenius extensions) in detail, with an aim towards applications to categorification.  Since, in categorical representation theory, one often wants to work in the setting of graded superrings, we adopt that generality here.  After recalling some conventions and basic facts about graded superrings in Section~\ref{sec:graded-superrings}, we introduce the notion of twisted Frobenius extensions in Section~\ref{sec:Frob-ext}.  The definition involves a certain isomorphism of bimodules.  In the case that all gradings are trivial (i.e.\ the non-graded non-super setting), our definition recovers that of \cite{Mor65}.  We prove that one has equivalent left and right formulations of the definition, and then examine the uniqueness of the data involved in the definition.  In Section~\ref{sec:Bil-Form} we give an alternative definition of twisted Frobenius extensions in terms of trace maps, bilinear forms, and dual generating sets.  We again answer the question of the uniqueness of this data.  In Section~\ref{sec:Nakayama} we define Nakayama automorphisms for twisted Frobenius extensions.  These generalize the usual Nakayama automorphisms for Frobenius extensions.

In Section~\ref{sec:adjoints} we examine one of the main motivations for the concept of twisted Frobenius extensions.  Namely, we show (Theorem~\ref{theo:twisted-adjointness}) that if $A$ is a graded superring and $B$ is a graded subring, then induction is twisted shifted right adjoint to restriction if and only if $A$ is a twisted Frobenius extension of $B$.  We conclude, in Section~\ref{sec:examples}, with some examples.  We prove (Corollary~\ref{cor:Frobalg-implies-Frobext}) that if $A$ is a Frobenius graded superalgebra, $B$ is a graded subalgebra of $A$ that is itself a Frobenius graded superalgebra, and $A$ is projective as a left $B$-module, then $A$ is a twisted Frobenius extension of $B$.  This explains the adjointness properties in this setting mentioned in the first paragraph of this introduction.  We conclude the paper with an example of a twisted Frobenius extension that is neither a usual Frobenius extension nor a Frobenius extension of the 2nd kind in the sense of \cite{NT60}.


\iftoggle{detailsnote}{
\medskip

\paragraph{\textbf{Note on the arXiv version}} For the interested reader, the tex file of the arXiv version of this paper includes hidden details of some straightforward computations and arguments that are omitted in the pdf file.  These details can be displayed by switching the \texttt{details} toggle to true in the tex file and recompiling.
}{}

\subsection*{Acknowledgements}

The authors would like to thank D.~Rosso for useful conversations and V.~Ostrik for pointing out the example of Remark~\ref{rem:non-proj-eg}.

%
\section{Graded superrings} \label{sec:graded-superrings}
%

Throughout this paper we fix an abelian group $\Lambda$ and by \emph{graded}, we mean $\Lambda$-graded.  In particular, a graded superring is a $\Lambda \times \Z_2$-graded ring.  So if $A$ is a graded superring, then we have
\[
  A = \bigoplus_{\mu \in \Lambda,\, \gamma \in \Z_2} A_{\mu,\gamma},\quad A_{\mu,\gamma} A_{\nu,\delta} \subseteq A_{\mu+\nu,\gamma+\delta},\quad \mu,\nu \in \Lambda,\, \gamma, \delta \in \Z_2.
\]
For the remainder of this section, $A$ will denote a graded superring.  To avoid repeated use of the modifiers ``graded'' and ``super'', from now on we will use the term \emph{ring} to mean graded superring and \emph{subring} to mean graded sub-superring.  Similarly, by an automorphism of a ring, we mean an automorphism as graded superrings.

We will use the term \emph{module} to mean graded supermodule.  In particular, a left $A$-module $M$ is a $(\Lambda \times\Z_2)$-graded abelian group with a left $A$-action such that
\[
  A_{\mu,\gamma} M_{\nu,\delta} \subseteq M_{\mu+\nu,\gamma+\delta}, \quad \mu,\nu \in \Lambda,\ \gamma,\delta \in \Z_2,
\]
and similarly for right modules.  If $v$ is a homogeneous element in a $\Lambda$-graded (resp.\ $\Z_2$-graded) ring or module, we will denote by $|v|$ (resp.\ $\bar v$) its degree.  Whenever we write an expression involving degrees of elements, we will implicitly assume that such elements are homogeneous.

For $M$, $N$ two $\Lambda \times \Z_2$-graded abelian groups, we define a $\Lambda \times \Z_2$-grading on the space $\HOM_\Z(M,N)$ of all $\Z$-linear maps by setting $\HOM_\Z(M,N)_{\mu,\gamma}$, $\mu \in \Lambda$, $\gamma \in \Z_2$, to be the subspace of all homogeneous maps of degree $(\mu,\gamma)$. That is,
\[
  \HOM_\Z(M,N)_{\mu,\gamma} := \{f \in \HOM_\Z(M,N)\ |\  f(M_{\nu,\delta}) \subseteq N_{\mu+\nu,\gamma+\delta} \ \forall\ \nu \in \Lambda,\ \delta \in \Z_2\}.
\]

We let $A\md$ denote the category of left $A$-modules, with set of morphisms from $M$ to $N$ given by
\[
  \Hom_A^\rL(M,N) = \{f \in \HOM_\Z(M,N)_{(0,0)}\ |\ f(am) = af(m) \ \forall\ a \in A,\ m \in M\}.
\]
Similarly, we let $A\rmd$ denote the category of right $A$-modules, with set of morphisms from $M$ to $N$ given by
\[
  \Hom_A^\rR(M,N) = \{f \in \HOM_\Z(M,N)_{(0,0)}\ |\ f(ma) = f(m)a \ \forall\ a \in A,\ m \in M\}.
\]
For $\mu \in \Lambda$ and $\gamma \in \Z_2$, we have the shift functors
\begin{gather*}
  \{\mu,\gamma\} \colon A\md \to A\md,\quad M \mapsto \{\mu,\gamma\}M, \\
  \{\mu,\gamma\} \colon A\rmd \to A\rmd,\quad M \mapsto \{\mu,\gamma\}M.
\end{gather*}
Here $\{\mu,\gamma\}M$ denotes the $\Lambda \times \Z_2$-graded abelian group that has the same underlying abelian group as $M$, but a new grading given by $(\{\mu,\gamma\}M)_{\nu,\delta}=M_{\nu-\mu,\delta+\gamma}$.  Abusing notation, we will also sometimes use $\{\mu,\gamma\}$ to denote the map $M \to \{\mu,\gamma\}M$ that is the identity on elements of $M$.  For $M \in A\md$, the left action of $A$ on $\{\mu,\gamma\}M$ is given by $a\cdot \{\mu,\gamma\} m= (-1)^{\gamma \bar a} \{\mu,\gamma\}am$, $a \in A$, $m \in M$, where $am$ is the action on $M$.  For $A \in A\rmd$, the right action of $A$ on $\{\mu,\gamma\}M$ is given by $(\{\mu,\gamma\} m) \cdot a = \{\mu,\gamma\}(ma)$, $a \in A$, $m \in M$.  The shift functors leave morphisms unchanged.  Note that we have chosen to write the shift $\{\mu,\gamma\}$ on the left since it commutes with right actions, but not left actions.

For $M,N\in A\md$, we also define the $\Lambda \times\Z_2$-graded abelian group
\[
  \HOM_A^\rL(M,N) := \bigoplus_{\mu \in \Lambda, \gamma \in \Z_2} \HOM_A^\rL(M,N)_{\mu, \gamma},
\]
where the homogeneous components are defined by
\begin{equation} \label{bighom} \ts
  \HOM_A^\rL(M,N)_{\mu,\gamma} := \{f \in \HOM_\Z(M,N)_{\mu,\gamma}\ |\ f(am)=(-1)^{\gamma \bar a} a f(m) \ \forall\ a \in A,\ m\in M\}.
\end{equation}
Note that we have an isomorphism of $\Lambda \times \Z_2$-graded abelian groups (which we will often view as identification)
\begin{equation} \label{eq:HOM-isom}
  \HOM_A^\rL(M,N)_{\mu, \gamma} \cong \HOM_A^\rL(M, \{\nu,\delta\}N)_{\mu+\nu,\gamma+\delta}.
\end{equation}
In particular, taking $\nu = -\mu$ and $\delta = \gamma$ in \eqref{eq:HOM-isom} gives the isomorphism
\begin{equation} \label{eq:Lhom-shift-identification}
  \HOM_A^\rL(M,N)_{\mu, \gamma} \cong \Hom_A^\rL(M, \{-\mu,\gamma\}N).
\end{equation}

For right modules, we define
\[
  \HOM_A^\rR(M,N)_{\mu,\gamma} := \{f \in \HOM_\Z(M,N)_{\mu,\gamma}\ |\ f(ma)=f(m)a \ \forall\ a \in A,\ m\in M\},
\]
and we have isomorphisms of $\Lambda \times \Z_2$-graded abelian groups (which we will often view as identification)
\begin{equation} \label{eq:Rhom-shift-identification}
  \HOM_A^\rR(M,N)_{\mu, \gamma} \cong \Hom_A^\rR(M, \{-\mu,\gamma\}N).
\end{equation}

If $M$ is a left $A$-module, we will sometimes use the notation $\pl{a}$ for the operator given by the left action by $a$, that is,
\begin{equation} \label{eq:left-mult-action}
  \pl{a}(m) = am ,\quad a \in A,\ m \in M.
\end{equation}
If $M$ is a right $A$-module, then for each homogeneous $a \in A$, we define a $\Z$-linear operator
\begin{equation} \label{eq:sign-right-action}
  \pr{a} \colon M\to M,\quad \pr{a}(m):=(-1)^{\bar a \bar m} ma, \quad a \in A,\ m \in M.
\end{equation}

If $A_1$ and $A_2$ are rings, then an $(A_1,A_2)$-bimodule is an abelian group that is simultaneously a left $A_1$-module and a right $A_2$-module and such that
\[
  (a_1 m)a_2 = a_1(ma_2), \quad a_1 \in A_1,\ a_2 \in A_2,\ m \in M.
\]
If $M$ is an $(A_1,A_2)$-bimodule and $N$ is a left $A_1$-module, then $\HOM_{A_1}^\rL(M,N)$ is a left $A_2$-module via the action
\begin{equation} \label{eq:left-hom-action}
  a \cdot f = (-1)^{\bar a \bar f} f \circ \pr{a},\quad a \in A_2,\ f \in \HOM_{A_1}^\rL(M,N),
\end{equation}
and $\HOM_{A_1}^\rL(N,M)$ is a right $A_2$-module via the action
\begin{equation} \label{eq:right-hom-action}
  f \cdot a = (-1)^{\bar a \bar f} (\pr{a}) \circ f,\quad a \in A_2,\ f \in \HOM_{A_1}^\rL(N,M).
\end{equation}
It is routine to verify that~\eqref{eq:HOM-isom} is an isomorphism of left $A_2$-modules and right $A_2$-modules under the actions~\eqref{eq:left-hom-action} and~\eqref{eq:right-hom-action}, respectively, when $M$ and $N$ have the appropriate structure.

Similarly, if $M$ is an $(A_1,A_2)$-bimodule and $N$ is a right $A_2$-module, then $\HOM_{A_2}^\rR(M,N)$ is a right $A_1$-module via the action
\begin{equation} \label{eq:right-rhom-action}
  f \cdot a = f \circ \pl{a},\quad a \in A_1,\ f \in \HOM_{A_2}^\rR(M,N),
\end{equation}
and $\HOM_{A_2}^\rR(N,M)$ is a left $A_1$-module via the action
\begin{equation} \label{eq:left-rhom-action}
  a \cdot f = \pl{a} \circ f,\quad a \in A_1,\ f \in \HOM_{A_2}^\rR(N,M).
\end{equation}

The tensor product $A_1 \otimes A_2 = A_1 \otimes_\Z A_2$ is also a ring, with product
\[
  (a_1 \otimes a_2)(a_1' \otimes a_2')=(-1)^{\bar a_2 \bar a_1'} a_1a_1'\otimes a_2a_2', \quad a_1,\ a_1'\in A_1,\ a_2,\ a_2'\in A_2.
\]
Notice that there is an isomorphism $A_1 \otimes A_2 \cong A_2 \otimes A_1$ of rings defined by
\[
  a_1 \otimes a_2 \mapsto (-1)^{\bar a_1 \bar a_2} a_2 \otimes a_1, \quad a_1 \in A_1,\ a_2 \in A_2.
\]

\begin{defin}[Projective bases] \label{def:proj-basis}
  A \emph{left projective basis} of a left $A$-module $M$ is a pair of sets of homogeneous elements
  \[
    \{x_i \in M_{\mu_i,\gamma_i}\}_{i\in I} \quad \text{and} \quad \{\varphi_i \in   \HOM_A^\rL(M,A)_{-\mu_i,\gamma_i}\}_{i \in I},
  \]
  where $\mu_i \in \Lambda$ and $\gamma_i \in \Z_2$ for $i \in I$, such that, for any $x\in M$, $\varphi_i(x)$ is nonzero for only finitely many $i$, and $x=\sum_{i \in I} (-1)^{\bar x \gamma_i}\varphi_i(x) x_i$.

  A \emph{right projective basis} of a right $A$-module $M$ is a pair of sets of homogeneous elements
  \[
    \{x_i \in M_{\mu_i,\gamma_i}\}_{i\in I} \quad \text{and} \quad \{\varphi_i \in   \HOM_A^\rR(M,A)_{-\mu_i,\gamma_i}\}_{i \in I},
  \]
  where $\mu_i \in \Lambda$ and $\gamma_i \in \Z_2$ for $i \in I$, such that, for any $x\in M$, $\varphi_i(x)$ is nonzero for only finitely many $i$, and $x=\sum_{i \in I} x_i \varphi_i(x)$.
\end{defin}

\begin{lem} \label{lem:proj-basis}
  A left (resp.\ right) $A$-module $M$ is projective if and only if there exists a left (resp.\ right) projective basis of $M$.  Furthermore, if $M$ is projective, then it is finitely generated if and only if it has a finite left (resp.\ right) projective basis (i.e.\ one for which the index set $I$ in Definition~\ref{def:proj-basis} is finite).
\end{lem}

\begin{proof}
  We prove the left version of the lemma, since the right version is analogous, but with fewer signs. Let $\{x_i \in M_{\mu_i,\gamma_i}\}_{i \in I}$ and $\{\varphi_i \in \HOM_A^\rL(M,A)_{-\mu_i,\gamma_i}\}_{i \in I}$ be a left projective basis of $M$. Then the map
  \[
    g \colon \bigoplus_{i \in I} \{\mu_i,\gamma_i\}A \to M, \quad (\{\mu_i,\gamma_i\}a_i)_{i\in I} \mapsto \sum_{i\in I} (-1)^{\bar a_i \gamma_i}a_i x_i,
  \]
  is a surjective left $A$-module homomorphism since, for any $x\in M$, we have $x=g((\varphi_i(x))_{i \in I})$, where we have used the isomorphism \eqref{eq:Lhom-shift-identification} to identify $\HOM_A^\rL(M,A)_{-\mu_i, \gamma_i}$ with $\Hom_A^\rL(M, \{\mu_i,\gamma_i\}A$).
  \details{For any $a \in A$ we have
    \begin{multline*}
      g(a \cdot(\{\mu_i,\gamma_i\}a_i)_{i\in I})=g((a \cdot \{\mu_i,\gamma_i\}a_i)_{i \in I})=g((\{\mu_i,\gamma_i\}((-1)^{\bar a \gamma_i} a a_i))_{i \in I}) \\
      =\sum_{i \in I}(-1)^{\bar a_i \gamma_i} a a_ix_i= a g((\{\mu_i,\gamma_i\}a_i)_{i \in I}),
    \end{multline*}
    making $g$ a homomorphism of left $A$-modules.
  }
  We also have an $A$-module homomorphism
  \[
    \varphi \colon M \to \bigoplus_{i \in I} \{\mu_i,\gamma_i\}A, \quad x \mapsto (\varphi_i(x))_{i \in I}.
  \]
  Since $g \circ \varphi = \id_M$, $g$ is a split epimorphism, and hence $M$ is a direct summand of the free module $\bigoplus_{i\in I} \{\mu_i,\gamma_i\}A$.  Thus $M$ is projective.  The converse can be seen by reversing the argument.
  \details{
    Let $M$ be projective.  Then there exists a split epimorphism
    \[
      g\colon \bigoplus_{i \in I} \{\mu_i,\gamma_i\} A \to M,
    \]
    for some $I$ and $\mu_i \in \Lambda$, $\gamma_i \in \Z_2$, $i \in I$.  For each $i \in I$, set $x_i=g(\{\mu_i,\gamma_i\}1_i) \in M_{\mu_i,\gamma_i}$, where $\{\mu_i,\gamma_i\}1_i \in \bigoplus_{i \in I} \{\mu_i,\gamma_i\}A$ is the element with $\{\mu_i,\gamma_i\}1_A$ in the $i$-th component and zeroes elsewhere.  Then
    \[
      g((\{\mu_i,\gamma_i\}a_i)_{i\in I})
      = \sum_{i \in I} g((-1)^{\bar a_i\gamma_i}a_i\cdot \{\mu_i,\gamma_i\}1_i)
      = \sum_{i\in I}(-1)^{\bar a_i \gamma_i} a_i x_i.
    \]
    Since $g$ is a split epimorphism, there exists a left $A$-module homomorphism
    \[
      \varphi \colon M \to \bigoplus_{i \in I} \{\mu_i,\gamma_i\}A
    \]
    such that $g \circ \varphi =\id_M$.  For $i \in I$, let $\varphi_i$ be the $i$-th component of $\varphi$, viewed as a degree $(-\mu_i,\gamma_i)$ map to $A$.  Then $\varphi_i(x) = 0$ for all but finitely many $i \in I$ by the definition of direct sums, and
    \[
      x = g \circ \varphi(x) = \sum_{i \in I}(-1)^{\bar x \gamma_i} \varphi_i(x) x_i.
    \]
  }

  \details{
    For the right version, let $\{x_i \in M_{\mu_i,\gamma_i}\}_{i \in I}$ and $\{\varphi_i \in \HOM_A^\rR (M,A)_{-\mu_i,\gamma_i}\}_{i \in I}$ be a right projective basis for $M$. Then the map
    \[
       h\colon \bigoplus_{i \in I} \{\mu_i,\gamma_i\}A \to M,\quad (\{\mu_i,\gamma_i\}a_i)_{i \in I} \mapsto \sum_{i \in I} x_i a_i,
    \]
    is surjective since, for any $x \in M$, we have $x = h((\varphi_i(x))_{i \in I})$.  It is a homomorphism of right $A$-modules since, for $a \in A$, we have
    \[
      h((a_i)_{i \in I} \cdot a) = h((a_i a)_{i \in I}) = \sum_{i \in I} x_ia_ia = h((a_i)_{i \in I}) \cdot a.
    \]
    The remainder of the proof is analogous to the left version.
  }

  Now assume that $M$ is projective.  If $M$ has a finite left projective basis, then it is clearly finitely generated.   On the other hand, suppose $M$ is finitely generated and $\{x_i\}_{i \in I}$, $\{\varphi_i\}_{i \in I}$ is a left projective basis.  Choose a finite collection $y_1,\dotsc,y_n$ of generators of $M$.  Then is it easy to see that $\{x_i\}_{i \in J}$, $\{\varphi_i\}_{i \in J}$ is a finite left projective basis of $M$, where
  \[
    J = \{i \in I\ |\ \varphi_i(y_k) \ne 0 \text{ for some } k=1,2,\dotsc,n\}. \qedhere
  \]
\end{proof}

If $B$ is a subring of $A$, we define the \emph{centralizer} of $B$ in $A$ to be the subring of $A$ defined by
\begin{equation}
  C_A(B):= \{a \in A\ |\ ab=(-1)^{\bar a \bar b}ba \quad \forall\ b\in B\}.
\end{equation}

When we wish to consider a ring as a right and/or left module over some subring, we will denote this using subscripts.  For instance, if $B$ is a subring of $A$, then $\prescript{}{A}A_B$ denotes the ring $A$ considered as an $(A,B)$-bimodule.  In the remainder of the paper, we will also often use subscripts on modules to make it clear what type of module they are.  For example $\prescript{}{B}{M}_A$ will denote that $M$ is a $(B,A)$-bimodule.

Suppose $\prescript{}{A}M$ is a left $A$-module, $N_A$ is a right $A$-module, and $\alpha$ is a ring automorphism of $A$.  Then we can define the twisted left $A$-module $\prescript{\alpha}{A}{M}$ and twisted right $A$-module $N^\alpha_A$ to be equal to $\prescript{}{A}M$ and $N_A$, respectively, as graded abelian groups, but with actions given by
\begin{gather}
  \label{eq:left-twist} a \cdot m = \alpha(a) m,\quad a \in A,\ m \in \prescript{\alpha}{A}{M}, \\
  \label{eq:right-twist} n \cdot a = n \alpha(a),\quad a \in A,\ n \in N^\alpha_A,
\end{gather}
where juxtaposition denotes the original action of $A$ on $\prescript{}{A}M$ and $N_A$.  If $\alpha$ is a ring automorphism of $A$, and $B$ is a subring of $A$, then we will also use the notation $\prescript{\alpha}{B}{A}_A$ to denote the $(B,A)$-bimodule equal to $A$ as a graded abelian group, with right action given by multiplication, and with left action given by $b \cdot a = \alpha(b)a$ (where here juxtaposition is multiplication in the ring $A$), even though $\alpha$ is not a ring automorphism of $B$.  We use $\prescript{}{A}A^\alpha_B$ for the obvious right analogue.  By convention, when we consider twisted modules as above, operators such as $\pr{a}$ and $\pl{a}$ defined in \eqref{eq:left-mult-action} and \eqref{eq:sign-right-action} involve the right and left action (respectively) in the \emph{original} (i.e.\ untwisted) module.

%
\section{Twisted Frobenius extensions} \label{sec:Frob-ext}
%

In this section, we define our main objects of study: twisted Frobenius extensions.  We fix a ring $A$ and subring $B$.  We also let $\alpha$ and $\beta$ denote automorphisms of $A$ and $B$, respectively.

\begin{prop} \label{prop:Frob-LR-equivalence}
  Fix $\lambda \in \Lambda$ and $\pi \in \Z_2$. The set of conditions
  \begin{description}[style=multiline, labelwidth=0.7cm]
    \item[\namedlabel{L1}{L1}] $A$ is finitely generated and projective as a left $B$-module,
    \item[\namedlabel{L2}{L2}] $\prescript{}{A}A_B \cong \HOM^\rL_B \left( \prescript{\beta}{B}A_A^\alpha, \{\lambda,\pi\}\prescript{}{B}B_B \right)$ as $(A,B)$-bimodules,
  \end{description}
  is equivalent to the set of conditions
  \begin{description}[style=multiline, labelwidth=0.7cm]
    \item[\namedlabel{R1}{R1}] $A$ is finitely generated and projective as a right $B$-module,
    \item[\namedlabel{R2}{R2}] $\prescript{}{B}A_A \cong \HOM^\rR_B \left( \prescript{\alpha^{-1}}{A}A_B^{\beta^{-1}}, \{\lambda,\pi\}\prescript{}{B}B_B \right)$ as $(B,A)$-bimodules.
  \end{description}
\end{prop}

\begin{proof}
  We prove that conditions \ref{L1} and \ref{L2} imply \ref{R1} and \ref{R2}.  The reverse implication is analogous.  Assume \ref{L1} and \ref{L2} hold.  By \ref{L1}, there exists a left $B$-module $\prescript{}{B}M$ such that
  \[
    \prescript{}{B}A \oplus \prescript{}{B}M \cong \bigoplus_{i=1}^n \{\mu_i,\gamma_i\} \prescript{}{B}B \quad \text{for some } n \in \N,\ \mu_1,\dotsc,\mu_n \in \Lambda,\ \gamma_1,\dotsc,\gamma_n \in \Z_2.
  \]
  Then $\prescript{\beta}{B}A \oplus \prescript{\beta}{B}M \cong \bigoplus_{i=1}^n \{\mu_i,\gamma_i\} \prescript{\beta}{B}B \cong \bigoplus_{i=1}^n \{\mu_i,\gamma_i\}\prescript{}{B}B$ as left $B$-modules.  So we have isomorphisms of right $B$-modules
  \begin{multline*}
    \HOM^\rL_B (\prescript{\beta}{B}A, \prescript{}{B}B_B) \oplus \HOM^\rL_B(\prescript{\beta}{B}M, \prescript{}{B}B_B) \cong \HOM^\rL_B \left( \bigoplus_{i=1}^n \{\mu_i,\gamma_i\} \prescript{}{B}B, \prescript{}{B}B_B \right) \\
    \cong \bigoplus_{i=1}^n \HOM^\rL_B \left( \prescript{}{B}B, \{-\mu_i,\gamma_i\} \prescript{}{B}B_B \right) \cong \bigoplus_{i=1}^n \{\mu_i,\gamma_i\}B_B.
  \end{multline*}
  Thus $\HOM^\rL_B (\prescript{\beta}{B}A, \prescript{}{B}B_B)$ is projective as a right $B$-module.  Hence, by \ref{L2}, we see that \ref{R1} is satisfied.

  By \ref{L2} there exists an isomorphism of $(A,B)$-bimodules $\Phi \colon \prescript{}{A}A_B \cong \HOM_B^{\rL} (\prescript{\beta}{B}A_A^\alpha,\{\lambda,\pi\}\prescript{}{B}B_B)$.  Define
  \begin{equation}\label{eq:left-right-correspondence}
    \Psi \colon \prescript{}{B}A_A \to \HOM_B^{\rR} \left( \prescript{\alpha^{-1}}{A}A_B^{\beta^{-1}},\{\lambda,\pi\} \prescript{}{B}B_B \right),\quad
    \Psi(a)(x) = (-1)^{\bar a \bar x}\beta(\Phi(x)(a)),\ a, x \in A.
  \end{equation}
  It suffices to show that $\Psi$ is an isomorphism of $(B,A)$-bimodules.  The verification that it is a homomorphism of $(B,A)$-bimodules is straightforward.
  \details{
    Let $a,a',x \in A$, $b\in B$.  We identify the codomains of $\Phi$ and $\Psi$ with $\HOM_B^{\rL} (\prescript{\beta}{B}A_A^\alpha,\prescript{}{B}B_B)_{-\lambda,\pi}$ and $\HOM_B^{\rR} \left( \prescript{\alpha^{-1}}{A}A_B^{\beta^{-1}}, \prescript{}{B}B_B \right)_{-\lambda,\pi}$ respectively.  Then we have
    \begin{multline*}
      \Psi(a)(x\cdot b) = \Psi(a)(x\beta^{-1}(b)) = (-1)^{\bar a (\bar x + \bar b)}\beta \big( \Phi(x \beta^{-1}(b))(a) \big) = (-1)^{\bar a (\bar x + \bar b) + \bar b (\bar x + \pi)} \beta \left( \pr{\beta^{-1}(b)}\circ \Phi(x)(a)\right) \\
      = (-1)^{\bar a \bar x} \beta \big( \Phi(x)(a) \beta^{-1}(b) \big)= (-1)^{\bar a \bar x} \beta \big( \Phi(x)(a) \big) b = \Psi(a)(x)b,
    \end{multline*}
    and so $\Psi(a) \in \Hom_B^{\rR} \left( \prescript{\alpha^{-1}}{A}A_B^{\beta^{-1}}, \prescript{}{B}B_B \right)$.  Furthermore,
    \begin{multline*}
      \Psi(b\cdot a)(x) = (-1)^{(\bar b + \bar a)\bar x} \beta (\Phi(x)(ba)) = (-1)^{\bar a \bar x} \beta (\beta^{-1}(b)\Phi(x)(a)) \\
      = (-1)^{\bar a \bar x} b \beta (\Phi(x)(a))
      = b \Psi(a)(x)
      = (b \cdot \Psi(a))(x).
    \end{multline*}
    Thus, $\Psi(b \cdot a) = b \cdot \Psi(a)$ and so $\Psi$ is a homomorphism of left $B$-modules.  We also have
    \begin{multline*}
      \Psi(a\cdot a')(x)
      = (-1)^{(\bar a + \bar a') \bar x} \beta (\Phi(x)(aa'))
      = (-1)^{(\bar a + \bar a') \bar x + \bar a \bar a'} \beta (\Phi(x) \circ \pr{a'} (a)) \\
      = (-1)^{\bar a \bar x + \pi \bar a' + \bar a \bar a'} \beta \big( (\alpha^{-1}(a')\cdot \Phi(x))(a) \big)
      = (-1)^{\bar a (\bar x + \bar a')} \beta \big( \Phi(\alpha^{-1}(a')x)(a) \big) \\
      = \Psi(a)(\alpha^{-1}(a')x)
      = (\Psi(a)\cdot a')(x).
    \end{multline*}
    Thus $\Psi(a \cdot a') = \Psi(a) \cdot a'$ and  so $\Psi$ is a homomorphism of left $A$-modules.  The degree shift of $\{\lambda,\pi\}$ follows immediately from the degree shift of $\{\lambda,\pi\}$ present in the isomorphism $\Phi$.
  }
  To see that $\Psi$ is an isomorphism, consider the natural map
  \begin{multline*}
    \epsilon \colon \prescript{\beta}{B}A_A^\alpha \to \HOM^{\rR}_B\left( \HOM_B^{\rL} \left( \prescript{\beta}{B}A_A^\alpha,\prescript{}{B}B_B \right),\prescript{}{B}B_B\right) \\
    \cong \HOM^{\rR}_B\left( \HOM_B^{\rL} \left( \prescript{\beta}{B}A_A^\alpha,\{\lambda,\pi\}\prescript{}{B}B_B \right), \{\lambda,\pi\}\prescript{}{B}B_B \right),\quad a \mapsto (f\mapsto f(a)).
  \end{multline*}
  If $\Psi(a)=0$ then $\Phi(x)(a)=0$ for every $x\in A$.  Then, since $\Phi$ is an isomorphism, $f(a)=0$ for every $f \in \HOM_B^{\rL} (\prescript{\beta}{B}A_A^\alpha, \prescript{}{B}B_B)$, and hence $\epsilon(a)=0$. But  $\prescript{\beta}{B}A_A^\alpha$ is a finitely generated and projective left $B$-module, and so $\epsilon$ is an isomorphism.  We conclude that $a=0$, and thus $\Psi$ is a monomorphism.   Moreover, any element of $\HOM^{\rR}_B (\HOM_B^{\rL} ( \prescript{\beta}{B}A_A^\alpha, \{\lambda,\pi\}\prescript{}{B}B_B ),\{\lambda,\pi\} \prescript{}{B}B_B )$ is of the form $f \mapsto f(a)$ for some $a\in A$, and hence any element of $\HOM^{\rR}_B \left( \prescript{}{A}A_B, \{\lambda,\pi\}\prescript{}{B}B_B \right)$ is of the form $x \mapsto \Phi(x)(a)$, so that $\Psi$ is an epimorphism.
\end{proof}

Note that~\eqref{eq:left-right-correspondence} gives an explicit one-to-one correspondence between the isomorphisms in \ref{L2} and those in \ref{R2}.

\begin{defin}[Twisted Frobenius extension] \label{def:twisted-Frob-ext}
  We call $A$ an \emph{$(\alpha,\beta)$-Frobenius extension} of $B$ of \emph{degree} $(-\lambda,\pi)$ if \ref{L1} and \ref{L2} (equivalently, \ref{R1} \and \ref{R2}) are satisfied.  We call $A$ a \emph{twisted Frobenius extension} of $B$ if there exist $\alpha, \beta, \lambda, \pi$ such that $A$ is an $(\alpha,\beta)$-Frobenius extension of $B$ of degree $(-\lambda,\pi)$.
\end{defin}

\begin{rem}
  We say the extension is of degree $(-\lambda,\pi)$ since that is the degree of the bimodule isomorphisms \ref{L2} and \ref{R2} under the identifications \eqref{eq:Lhom-shift-identification} and \eqref{eq:Rhom-shift-identification}.  Note that, if $A$ and $B$ are concentrated in degree $(0,0)$, then we recover the $(\alpha,\beta)$-Frobenius extensions of \cite{Mor65}.  In particular, an $(\id_A,\beta)$-Frobenius extension of degree $(0,0)$ is a $\beta^{-1}$-extension in the language of \cite{NT60}.  Furthermore, an $(\id_A,\id_B)$-Frobenius extension of degree $(0,0)$ is a Frobenius extension in the usual sense.
  \details{A homomorphism $\varphi \in \Hom^\rR_B(M_B^{\beta},N_B)$ has the property $\varphi(x)b=\varphi(x\cdot b)=\varphi(x\beta(b))$, and hence $\varphi(xb)=\varphi(x)\beta^{-1}(b)$.  So there is a canonical isomorphism $\Hom^\rR_B(M_B^\beta,N_B)\cong \Hom^\rR_{B,\beta^{-1}}(M_B,N_B)$. Setting $\alpha=\id_A$ in condition \ref{R2} then gives $\prescript{}{B}A_A\cong \Hom^\rR_B(\prescript{}{A}A_B^{\beta},\prescript{}{B}B_B)\cong \Hom^\rR_{B,\beta^{-1}}(\prescript{}{A}A_B,\prescript{}{B}B_B)$, which is condition $2_r$ of \cite{NT60} in the definition of a $\beta^{-1}$ extension. }
  If $B=\F$ and $\Lambda = \Z$, then an $(\id_A,\id_\F)$-Frobenius extension of degree $(-\lambda,\pi)$ is a Frobenius $\Z$-graded superalgebra of degree $(-\lambda,\pi)$ (see \cite[Def.~6.1]{RS15}).  Thus, Definition~\ref{def:twisted-Frob-ext} generalizes several definitions appearing in the literature.
\end{rem}

\begin{prop} \label{prop:Frob-data-uniqueness}
  Let $A$ be an $(\alpha,\beta)$-Frobenius extension of $B$ of degree $(-\lambda,\pi)$ and let $\alpha'$, $\beta'$ be automorphisms of $A$, $B$ respectively.  Furthermore, let $\lambda' \in \Lambda$ and $\pi' \in \Z_2$.  Then $A$ is an $(\alpha',\beta')$-Frobenius extension of $B$ of degree $(-\lambda',\pi')$ if and only if there exists an invertible element $\mu \in A_{(\lambda'-\lambda, \pi'-\pi)}$ such that
  \begin{equation} \label{eq:Frob-data-uniqueness}
    \mu \left( (\alpha')^{-1} \alpha(b) \right) \mu^{-1} = (\beta')^{-1}\beta(b) \quad \text{ for all } b \in B.
  \end{equation}
\end{prop}

\begin{proof}
  Set $\varsigma = (\beta')^{-1}\beta$ and $\sigma = (\alpha')^{-1}\alpha$.  Then we have an isomorphism of $(A,B)$-bimodules
  \[
    \Gamma \colon \HOM^\rL_B (\prescript{\beta}{B}A^{\alpha}_A, \prescript{}{B}B_B) \to \prescript{\sigma}{}\HOM_B^{\rL}(\prescript{\beta'}{B}A^{\alpha'}_{A}, \prescript{}{B}B_B)^{\varsigma}, \quad
    \varphi \mapsto \varsigma \circ \varphi.
  \]
  \details{
    For $\varphi \in \Hom^\rL_B (\prescript{\beta}{B}A^{\alpha}_A, \prescript{}{B}B_B)$, we have $\Gamma(\varphi)\in \prescript{\varsigma}{}\Hom_B^{\rL}(\prescript{\beta'}{B}A^{\alpha'}_{A},\prescript{}{B}B_B)^{\sigma}$ since, for all $x \in \prescript{\beta}{B}A^{\alpha}_{A}$,
    \[
      \varsigma \circ \varphi(b\cdot x) = \varsigma \circ \varphi (\beta'(b)x) = \varsigma(\beta^{-1} \beta'(b) \varphi(x)) = b \varsigma \varphi(x).
    \]
    The map $\Gamma$ is a homomorphism of right $B$-modules since
    \[
      \Gamma(\varphi \cdot b) = (-1)^{\bar b \bar \varphi} \Gamma \left( \pr{b} \circ \varphi \right) = (-1)^{\bar b \bar \varphi} \varsigma \circ \pr{b} \circ \varphi = (-1)^{\bar b \bar \varphi} \left( \pr{\varsigma(b)} \right) \circ \varsigma \circ \varphi = (-1)^{\bar b \bar \varphi} \left( \pr{\varsigma(b)} \right) \circ \Gamma(\varphi) = \Gamma(\varphi) \cdot b.
    \]
    It is homomorphism of left $A$-modules since
    \[
      \Gamma(a \cdot \varphi) = (-1)^{\bar a \bar \varphi} \Gamma \left(\varphi \circ \pr{\alpha(a)} \right) = (-1)^{\bar a \bar \varphi} \varsigma \circ \varphi \circ \pr{\alpha(a)} = (-1)^{\bar a \bar \varphi} \varsigma \circ \varphi \circ \pr{\alpha'(\sigma (a))}= a \cdot (\varsigma \circ \varphi) = a \cdot \Gamma(\varphi).
    \]
    It is clearly invertible with inverse $\varphi \mapsto \varsigma^{-1} \circ \varphi$.
  }
  Also, since $A$ is an $(\alpha,\beta)$-Frobenius extension of $B$ of degree $(-\lambda,\pi)$, there is an isomorphism of $(A,B)$-bimodules $\Phi \colon \prescript{}{A}A_B \to \HOM^{\rL}_B (\prescript{\beta}{B}A^{\alpha}_A, \{\lambda,\pi\}\prescript{}{B}B_B)$.

  Now suppose that $A$ is also an $(\alpha',\beta')$-Frobenius extension of degree $(-\lambda',\pi')$.  Then we have an $(A,B)$-bimodule isomorphism $\Phi' \colon \prescript{}{A}A_B \to \Hom^{\rL}_B (\prescript{\beta'}{B}A^{\alpha'}_A, \{\lambda', \pi'\} \prescript{}{B}B_B)$.  This induces another isomorphism of $(A,B)$-bimodules $\Phi'_{\varsigma\sigma} \colon \prescript{\sigma}{A}A_B^{\varsigma} \to \prescript{\sigma}{}\Hom^{\rL}_B (\prescript{\beta'}{B}A^{\alpha'}_A, \{\lambda',\pi'\} \prescript{}{B}B_B)^\varsigma$.  Consider the $(A,B)$-bimodule isomorphism
  \[
    \rho := (\Phi_{\varsigma\sigma}')^{-1} \Gamma \Phi \colon \prescript{}{A}A_B \xrightarrow{\cong} \{\lambda-\lambda',\pi-\pi'\} \prescript{\sigma}{A}A^\varsigma_B.
  \]
  We claim that $\rho(1_A)$ is an invertible element in $A$.  Indeed, for any $x \in A$, we have $x\rho(1_A) = \rho(\sigma^{-1}(x)1_A) = \rho(\sigma^{-1}(x))$.  Thus $A\rho(1_A) = \rho\sigma^{-1}(A)=A$, and if $x\rho(1_A)=0$ then $x=0$, which shows that right multiplication by $\rho(1_A)$ is a left $A$-module automorphism of $A$.  Hence $\rho(1_A)$ has a left inverse $a$ in $A$. Then
  \[
    (1_A - \rho(1_A)a)\rho(1_A) = 0_A \implies 1_A = \rho(1_A)a,
  \]
  and so $a$ is also a right inverse to $\rho(1_A)$.

  Finally, for any $b \in B$, we have
  \[
    \rho(1_A) \varsigma(b) \rho(1_A)^{-1} = \rho(1_A b) \rho(1_A)^{-1} = \rho(b 1_A) \rho(1_A)^{-1} = \sigma(b) \rho(1_A) \rho(1_A)^{-1} = \sigma(b),
  \]
  so setting $\mu=\rho(1_A)^{-1}$ gives~\eqref{eq:Frob-data-uniqueness}.

  Conversely, suppose there is some invertible element $\mu \in A_{(\lambda'-\lambda, \pi'-\pi)}$ such that~\eqref{eq:Frob-data-uniqueness} is satisfied, that is, such that $\mu \sigma(b) \mu^{-1} = \varsigma(b)$ for all $b \in B$.  Then the map
  \[
    \rho \colon \prescript{}{A}A_B \to \{\lambda-\lambda', \pi-\pi'\} \prescript{\sigma}{A}A^\varsigma_B,\quad \rho(x)=\sigma(x) \mu^{-1},
  \]
  is an isomorphism of $(A,B)$-bimodules.
  \details{
    The map is clearly invertible.  For $a, x \in A$ and $b \in B$, we have
    \begin{gather*}
      \rho(a \cdot x) = \rho(ax) = \sigma(a x) \mu^{-1} = \sigma(a) \sigma(x) \mu^{-1} = \sigma(a) \rho(x), \quad \text{and} \\
      \rho(x \cdot b) = \rho(x b) = \sigma(x b) \mu^{-1} = \sigma(x) \sigma(b) \mu^{-1} = \sigma(x) \mu^{-1} \varsigma(b) = \rho(x) \cdot b.
    \end{gather*}
  }
  We then have an isomorphism of $(A,B)$-bimodules
  \[
    \Gamma \Phi \rho^{-1} \colon \prescript{\sigma}{A}A_B^\varsigma \cong \prescript{\sigma}{}\HOM^{\rL}_B(\prescript{\beta'}{B}A^{\alpha'}_A,\{\lambda',\pi'\} \prescript{}{B}B_B )^\varsigma .
  \]
  It follows that we have $\prescript{}{A}A_B\cong \HOM^{\rL}_B(\prescript{\beta'}{B}A^{\alpha'}_A, \{\lambda',\pi'\} \prescript{}{B}B_B)$, and so $A$ is an $(\alpha',\beta')$-Frobenius extension of degree $(-\lambda',\pi')$.
\end{proof}

\begin{rem}
  If the abelian group $\Lambda$ is free, then all invertible elements of $A$ lie in $\Lambda$-degree zero.  This forces $\lambda=\lambda'$ in Proposition~\ref{prop:Frob-data-uniqueness}.  If $\alpha = \alpha' = \id_A$ and all gradings are trivial (that is, $A$ is concentrated in degree $(0,0)$), then Proposition~\ref{prop:Frob-data-uniqueness} reduces to \cite[Prop.~2]{NT60}.
\end{rem}

\begin{cor} \label{cor:Frob-Nakayama-case}
  If $A$ is an $(\alpha,\beta)$-Frobenius extension of $B$ of degree $(-\lambda,\pi)$ and $\alpha(B)=B$, then $A$ is an $(\id_A,\beta \circ \alpha^{-1})$-Frobenius extension of $B$ of degree $(-\lambda,\pi)$.  In particular, if $A$ is concentrated in degree $(0,0)$, then $A$ is a $\beta \circ \alpha^{-1}$-extension of $B$ in the language of \cite{NT60}.
\end{cor}

\begin{proof}
  This follows immediately from Proposition~\ref{prop:Frob-data-uniqueness}, taking $\mu=1_A$, $\alpha' = \id_A$ and $\beta' = \beta \circ \alpha^{-1}$.
\end{proof}

\begin{rem}
  The converse to Corollary~\ref{cor:Frob-Nakayama-case} is false.  For example, suppose $\F$ is a field and let $S_n$ denote the symmetric group on $n$ elements.  Then $\F S_n$ is a Frobenius algebra with Nakayama automorphism $\psi_n$ given by conjugation by the longest element.  It follows from Corollary~\ref{cor:Frobalg-implies-Frobext} that $\F S_n$ is a $(\psi_n,\psi_m)$-Frobenius extension of $\F S_m$ for $m < n$.  Note that $\psi_n(s_1)=s_{n-1}$, and hence $\psi_n$ does not fix $\F S_m$.  However, if we take $\mu$ to be the longest element of $\F S_n$, $\alpha'=\id_A$, and $\beta'=\psi_m$ in Proposition~\ref{prop:Frob-data-uniqueness}, we see that $\F S_n$ is an $(\id_A, \psi_m)$-Frobenius extension of $\F S_m$.  Nonetheless, there do exist extensions $A$ of $B$ which are $(\alpha,\beta)$-Frobenius for some $\alpha,\beta$, but are not $(\id_A,\beta')$-extensions for any automorphism $\beta'$ of $B$ (see Example \ref{eg:Nilcoxeter}).
\end{rem}

\begin{prop} \label{prop:Frob-data-uniqueness2}
  Let $A$ be an $(\alpha,\beta)$-Frobenius extension of $B$ and let
  \[
    \Phi \colon \prescript{}{A}A_B \cong \Hom^\rL_B \left( \prescript{\beta}{B}A_A^\alpha, \{\lambda,\pi\} \prescript{}{B}B_B \right)
  \]
  be an isomorphism of $(A,B)$-bimodules.  Then a second map
  \[
    \Phi' \colon \prescript{}{A}A_B \to \Hom^\rL_B \left( \prescript{\beta}{B}A_A^\alpha, \{\lambda,\pi\} \prescript{}{B}B_B \right)
  \]
  is an isomorphism of $(A,B)$-bimodules if and only if there is an invertible element $\mu \in A_{0,0}$ that commutes with all elements of $B$, and such that
  \[
    \Phi^{-1} \Phi' (a) = a \mu \quad \text{for all } a \in A.
  \]
\end{prop}

\begin{proof}
  Consider the proof of Proposition~\ref{prop:Frob-data-uniqueness} with $\alpha=\alpha'$ and $\beta=\beta'$, so that $\varsigma = \id_B$ and $\sigma = \id_A$.  Then the map $\Gamma$ in the proof of Proposition~\ref{prop:Frob-data-uniqueness} is the identity map.  Moreover, it is shown there that $\rho(1_A)$ is invertible and that, for all $b \in B$, we have $ \rho(1_A) \varsigma(b)=\sigma(b)\rho(1_A)$, and hence $\rho(1_A)$ commutes with all elements of $B$.  Finally, we have
  \[
    a\rho(1_A)^{-1} = a\Phi^{-1} \Phi'(1_A) = \Phi^{-1}\Phi'(a),
  \]
  and so taking $\mu=\rho(1_A)^{-1}$ completes the proof of the proposition.
\end{proof}

\begin{rem}
  In the case that $\alpha = \id_A$ and all gradings are trivial, Proposition~\ref{prop:Frob-data-uniqueness2} can be found in \cite[Prop.~3]{NT60}.
\end{rem}

%
\section{Trace characterization} \label{sec:Bil-Form}
%

In this section, we give an alternative characterization of twisted Frobenius extensions in terms of trace maps and bilinear forms.  We continue to let $A$ denote a ring and $B$ denote a subring of $A$.  We let $\alpha$ and $\beta$ denote automorphisms of $A$ and $B$, respectively, and fix $\lambda \in \Lambda$ and $\pi \in \Z_2$.

Recall the one-to-one correspondence between isomorphisms in \ref{L2} and those in \ref{R2} given by~\eqref{eq:left-right-correspondence}.  Note that for any $a_1,a_2 \in A$, we have $\beta^{-1}\Psi(a_1)(a_2) = (-1)^{\bar a_1 \bar a_2}\Phi(a_2)(a_1)$.  On the other hand, we have
\begin{gather*}
  \Psi(a_1)(a_2) = (\Psi(1_A)\cdot a_1)(a_2) = \Psi(1_A)(a_1\cdot a_2)=\Psi(1_A)(\alpha^{-1}(a_1)a_2), \quad \text{and} \\
  \Phi(a_2)(a_1) = (a_2\cdot \Phi(1_A))(a_1) = (-1)^{\bar a_1 \bar a_2} \Phi(1_A)(a_1\cdot a_2) = (-1)^{\bar a_1 \bar a_2} \Phi(1_A)(a_1\alpha(a_2)).
\end{gather*}
Therefore
\begin{equation}\label{eq:Bil-Form}
  \beta^{-1}\Psi(1_A)(\alpha^{-1}(a_1)a_2) = \Phi(1_A)(a_1\alpha(a_2)) \quad \text{for all } a_1,a_2 \in A.
\end{equation}

\begin{prop} \label{prop:left-trace-map}
  If
  \[
    \Phi\colon \prescript{}{A}A_B \to \HOM_B^{\rL} \left( \prescript{\beta}{B}A_A^\alpha, \{\lambda,\pi\} \prescript{}{B}B_B \right)
  \]
  is an isomorphism of $(A,B)$-bimodules, then the map
  \[
    \tr^\rL := \Phi(1_A) \colon \prescript{\beta}{B}A^\alpha_B \to \{\lambda,\pi\} \prescript{}{B}B_B
  \]
  is a homomorphism of $(B,B)$-bimodules satisfying:
  \begin{description}[style=multiline, labelwidth=0.7cm]
    \item[\namedlabel{L3}{L3}] if $\tr^\rL(Aa)=0$ for some $a\in A$ then $a=0$,
    \item[\namedlabel{L4}{L4}] for every $\varphi \in \HOM_B^{\rL} (\prescript{\beta}{B}A, \{\lambda,\pi\}\prescript{}{B}B)$, there exists an $a \in A$ such that $\varphi = \tr^\rL \circ \pr{a}$.
  \end{description}

  Conversely, if there is a homomorphism of $(B,B)$-bimodules
  \[
    \tr^\rL \colon \prescript{\beta}{B}A^\alpha_B \to \{\lambda,\pi\} \prescript{}{B}B_B
  \]
  satisfying \ref{L3} and \ref{L4}, then $\tr^\rL= \Phi(1_A)$ for some $(A,B)$-bimodule isomorphism
  \[
    \Phi\colon \prescript{}{A}A_B \cong \Hom_B^{\rL} \left( \prescript{\beta}{B}A_A^\alpha, \{\lambda,\pi\} \prescript{}{B}B_B \right).
  \]
\end{prop}

\begin{proof}
  Since $\Phi(1_A) \in \Hom_B^{\rL}(\prescript{\beta}{B} A^\alpha_A, \{\lambda,\pi\} \prescript{}{B}B_B ) $, $\tr^\rL$ is a homomorphism of left $B$-modules.  Now, for $b \in B$ and $x\in A$, we have
  \begin{multline*}
    \tr^\rL(x\cdot b) = \Phi(1_A)(x \alpha(b)) = \beta^{-1} \Psi(1_A) (\alpha^{-1}(x)b)
    = \beta^{-1} (\Psi(1_A) (\alpha^{-1}(x))\beta(b)) \\
    = \beta^{-1} (\Psi(1_A) (\alpha^{-1}(x))) b
    = \Phi(1_A)(x)b
    = \tr^\rL(x) b,
  \end{multline*}
  where we used~\eqref{eq:Bil-Form} in the second and fifth equalities.  Thus, $\tr^\rL$ is also a homomorphism of right $B$-modules.

  If $\tr^\rL(Aa)=0$ for some $a \in A$ then $0 = \Phi(1_A)(Aa)= \pm \Phi(\alpha^{-1}(a))(A)$, and hence $\alpha^{-1}(a)=0$, so that $a=0$.  Thus $\tr^\rL$ satisfies property \ref{L3}.  Now let $\varphi \in \HOM_B^{\rL} (\prescript{\beta}{B}A, \{\lambda, \pi\} \prescript{}{B}B) = \HOM_B^{\rL} (\prescript{\beta}{B}A_A^\alpha, \{\lambda, \pi\} \prescript{}{B}B_B)$ (equality of abelian groups).  Then there exists $a' \in A$ such that $\varphi=\Phi(a')$.  Thus, for $x \in A$, we have
  \[
    \varphi = \Phi(a') = a' \cdot \Phi(1_A) = \Phi(1_A) \circ \pr{\alpha(a')} = \tr^\rL \circ \pr{\alpha(a')},
  \]
  and so taking $a=\alpha(a')$ shows that $\tr^\rL$ satisfies property \ref{L4}.

  On the other hand, let $\tr^\rL\colon \prescript{\beta}{B}A^\alpha_B \to \{\lambda, \pi\} \prescript{}{B}B_B $ be a homomorphism of $(B,B)$-bimodules satisfying \ref{L3} and \ref{L4}.  Define a map
  \[
    \Phi \colon \prescript{}{A}A_B \cong \HOM^\rL_B \left( \prescript{\beta}{B}A_A^\alpha, \{\lambda,\pi\} \prescript{}{B}B_B \right),\quad \Phi(a) = \tr^\rL \circ \pr{\alpha(a)}.
  \]
  The map $\Phi$ is clearly a homomorphism of left $A$-modules, and it is a homomorphism of right $B$-modules since $\tr^\rL$ is.
  \details{
    For $a,x \in A$ and $b \in B$, we have
    \[
      \Phi(a \cdot x) = \Phi(ax) = \tr^\rL \circ \pr{\alpha(ax)} = (-1)^{\bar x \bar a} \tr^\rL \circ \pr{\alpha(x)} \circ \pr{\alpha(a)} = (-1)^{\bar x \bar a} \Phi(x) \circ \pr{\alpha(a)} = a \cdot \Phi(x)
    \]
    and
    \begin{multline*}
      \Phi(x \cdot b) = \Phi(xb) = \tr^\rL \circ \pr{\alpha(xb)} = (-1)^{\bar x \bar b} \tr^\rL \circ \pr{\alpha(b)} \circ \pr{\alpha(x)} \\
      = (-1)^{(\bar x + \pi) \bar b} \big( \pr{b} \big) \circ \tr^\rL \circ \pr{\alpha(x)} =  (-1)^{(\bar x + \pi) \bar b} \big( \pr{b} \big) \circ \Phi(x) = \Phi(x) \cdot b,
    \end{multline*}
    where we have viewed $\tr^\rL$ and $\Phi$ as degree $(-\lambda,\pi)$ maps via the isomorphisms \eqref{eq:Lhom-shift-identification} and \eqref{eq:Rhom-shift-identification}.  Hence $\Phi$ is a homomorphism of $(A,B)$-bimodules.
  }
  Also, $\Phi$ is injective by \ref{L3} and surjective by \ref{L4}, and hence is an isomorphism.   It is clear that $\Phi(1_A)=\tr^\rL$.
\end{proof}

\begin{cor} \label{cor:left-trace-map}
  Assuming that \ref{L1} holds, $A$ is an $(\alpha,\beta)$-Frobenius extension of $B$ of degree $(-\lambda,\pi)$ if and only if there exists a homomorphism of $(B,B)$-bimodules $\tr^\rL\colon \prescript{\beta}{B}A^\alpha_B \to \{\lambda,\pi\} \prescript{}{B}B_B$ satisfying \ref{L3} and \ref{L4}.
\end{cor}

The following is the right-module analogue of Proposition~\ref{prop:left-trace-map}.

\begin{prop}\label{prop:right-trace-map}
  If
  \[
    \Psi \colon \prescript{}{B}A_A \to \HOM_B^{\rR} \left( \prescript{\alpha^{-1}}{A}A_B^{\beta^{-1}}, \{\lambda,\pi\} \prescript{}{B}B_B \right)
  \]
  is an isomorphism of $(B,A)$-bimodules, then the map
  \[
    \tr^\rR := \Psi(1_A) \colon \prescript{\alpha^{-1}}{B}A^{\beta^{-1}}_B \to \{\lambda,\pi\} \prescript{}{B}B_B
  \]
  is a homomorphism of $(B,B)$-bimodules satisfying:
  \begin{description}[style=multiline, labelwidth=0.7cm]
    \item[\namedlabel{R3}{R3}] if $\tr^\rR(aA)=0$ for some $a\in A$ then $a=0$,
    \item[\namedlabel{R4}{R4}] for every $\varphi \in \HOM_B^{\rR} (A_B^{\beta^{-1}},B_B )$, there exists an $a \in A$ such that $\varphi = \tr^\rR \circ \pl{a}$.
  \end{description}

  Conversely, if there is a homomorphism of $(B,B)$-bimodules
  \[
    \tr^\rR \colon \prescript{\alpha^{-1}}{B}A^{\beta^{-1}}_B \to \{\lambda,\pi\} \prescript{}{B}B_B
  \]
  satisfying \ref{R3} and \ref{R4}, then $\tr^\rR = \Psi(1_A)$ for some $(B,A)$-bimodule isomorphism
  \[
    \Psi \colon \prescript{}{B}A_A \to \Hom_B^{\rR} \left( \prescript{\alpha^{-1}}{A}A_B^{\beta^{-1}}, \{\lambda,\pi\} \prescript{}{B}B_B \right).
  \]
\end{prop}

\begin{proof}
  The proof is similar to that of Proposition~\ref{prop:left-trace-map} and will therefore be omitted.
\end{proof}

\begin{cor}\label{cor:right-trace-map}
  Assuming that \ref{R1} holds, $A$ is an $(\alpha,\beta)$-Frobenius extension of $B$ of degree $(-\lambda,\pi)$ if and only if there exists a homomorphism of $(B,B)$-bimodules $\tr^\rR\colon \prescript{\alpha^{-1}}{B}A^{\beta^{-1}}_B \to \{\lambda,\pi\} \prescript{}{B}B_B$ satisfying \ref{R3} and \ref{R4}.
\end{cor}

\begin{defin}[Trace map]
  A homomorphism of $(B,B$)-bimodules
  \[
    \prescript{\beta}{B}{A}^\alpha_B \to \{\lambda,\pi\} \prescript{}{B}B_B \quad \left( \text{resp.\ }
    \prescript{\alpha^{-1}}{B}A_B^{\beta^{-1}} \to \{\lambda,\pi\} \prescript{}{B}B_B \right)
  \]
  satisfying \ref{L3} and \ref{L4} (resp.\ \ref{R3} and \ref{R4}) will be referred to as a \emph{left} (resp.\ \emph{right}) \emph{$(\alpha,\beta)$-trace map of degree $(-\lambda,\pi)$}.  We will often simply use the terms \emph{left trace map} and \emph{right trace map} when $\alpha$, $\beta$, $\lambda$, and $\pi$ are clear from the context, or when we do not wish to specify them.  We use the term \emph{trace map} to mean left trace map.  In such cases, we will sometimes use the notation $\tr$ instead of $\tr^\rL$.
\end{defin}

\begin{lem}\label{lem:left-right-trace-map}
  Let $A$ be an $(\alpha,\beta)$-Frobenius extension of $B$ of degree $(-\lambda,\pi)$.  Then a map $\tr^\rR\colon \prescript{\alpha^{-1}}{B}A_B^{\beta^{-1}} \to \{\lambda,\pi\} \prescript{}{B}B_B$ is a right $(\alpha,\beta)$-trace map of degree $(-\lambda,\pi)$ if and only if $\tr^\rR = \beta \circ \tr^\rL \circ \alpha$ for some left $(\alpha,\beta)$-trace map $\tr^\rL$ of degree $(-\lambda,\pi)$.
\end{lem}

\begin{proof}
  Suppose $\tr^\rL$ is a left trace map, and define $\tr^\rR = \beta \circ \tr^\rL \circ \alpha$.  By Proposition \ref{prop:left-trace-map}, there is some isomorphism of $(A,B)$-bimodules  $\Phi\colon \prescript{}{A}A_B \to \Hom_B^{\rL} (\prescript{\beta}{B}A_A^\alpha, \{\lambda,\pi\} \prescript{}{B}B_B ) $ such that $\tr^\rL=\Phi(1_A)$.  Setting $a_2 = 1_A$ in equation~\eqref{eq:Bil-Form}, we see that $\tr^\rL= \beta^{-1} \circ \Psi(1_A) \circ \alpha^{-1}$ for some isomorphism of $(B,A)$-bimodules $ \Psi \colon \prescript{}{B}A_A \to \Hom_B^{\rR}(\prescript{\alpha^{-1}}{A}A_B^{\beta^{-1}},\prescript{}{B}B_B)$.  Then by Proposition~\ref{prop:right-trace-map}, $\Psi(1_A) = \beta \circ \tr^\rL \circ \alpha$ is a right trace map.  The proof of the reverse implication is analogous.
  \details{
    Suppose $\tr^\rR$ is a right trace map.  Then, by Proposition \ref{prop:right-trace-map}, there is an isomorphism of $(B,A)$-bimodules $ \Psi \colon \prescript{}{B}A_A \to \Hom_B^{\rR} (\prescript{\alpha^{-1}}{A}A_B^{\beta^{-1}}, \{\lambda,\pi\} \prescript{}{B}B_B)$ with $\tr^\rR=\Psi(1_A)$.  Setting $a_2 = 1_A$ in equation~\eqref{eq:Bil-Form} then gives $\tr^\rR=\beta \circ \Phi(1_A) \circ \alpha$, and noting that $\Phi(1_A)$ is a left $(\alpha,\beta)$-trace map by Proposition \ref{prop:left-trace-map} completes the proof.
  }
\end{proof}

We note that, thanks to Lemma~\ref{lem:left-right-trace-map}, any left trace map $\tr^\rL \colon \prescript{\beta}{B}A^\alpha_B \to \{\lambda,\pi\} \prescript{}{B}B_B$ will also satisfy property \ref{R3}.  Moreover, for every $\varphi \in \Hom_B^{\rR} (A_B^{\beta^{-1}}, \{\lambda,\pi\} B_B)$, there exists an $a \in A$ such that $\varphi = \beta \circ \tr^\rL \circ \pl{a}$.  Similarly,  any right trace map $\tr^\rR \colon \prescript{\alpha^{-1}}{B}A^{\beta^{-1}}_B \to \{\lambda,\pi\} \prescript{}{B}B_B $ will also satisfy property \ref{L3}.  Moreover, for every $\varphi \in \Hom_B^{\rL} (\prescript{\beta}{B}A, \{\lambda,\pi\} \prescript{}{B}B)$, there exists an $a \in A$ such that $\varphi =\beta^{-1} \circ \tr^\rR \circ \pr{a}$.

\begin{prop} \label{prop:trace-map-uniqueness}
  Let $A$ be an $(\alpha,\beta)$-Frobenius extension of $B$ of degree $(-\lambda,\pi)$ and let $\tr \colon \prescript{\beta}{B}A^\alpha_B \to \{\lambda,\pi\} \prescript{}{B}B_B$ be a trace map.  Then a second map $\tr' \colon \prescript{\beta}{B}A^\alpha_B \to \{\lambda,\pi\} \prescript{}{B}B_B$ is also a trace map if and only if there exists an invertible element $a \in C_A(\alpha(B))$ of degree $(0,0)$ such that $\tr' = \tr \circ \pr{a}$.
\end{prop}

\begin{proof}
  This follows from Propositions~\ref{prop:Frob-data-uniqueness2} and~\ref{prop:left-trace-map}.
  \details{
    By Proposition~\ref{prop:left-trace-map}, there exist isomorphisms of $(A,B)$-bimodules $\Phi$, $\Phi'$ with $\tr = \Phi(1_A)$ and $\tr' = \Phi'(1_A)$.  By Proposition~\ref{prop:Frob-data-uniqueness2}, there exists some invertible $\mu \in A_{0,0}$ that commutes with every element of $B$ and such that $\Phi^{-1}\Phi'(a)=a\mu$ for all $a \in A$.  Then, if $x\in A$, we have
    \[
      \tr'(x) = \Phi'(1_A)(x) = \Phi(\mu)(x) = \Phi(\mu\cdot 1_A)(x) = (-1)^{\bar x \bar \mu} \Phi(1_A)(x\cdot \mu) = (-1)^{\bar x \bar \mu} \Phi(1_A)(x \alpha(\mu)) = (-1)^{\bar x \bar \mu} \tr(x\alpha(\mu)).
    \]
    Setting $a = \alpha(\mu)$, we obtain the forward implication.

    Conversely, defining a map $\tr':= \tr \circ \pr{\alpha(\mu)}$ for some $\mu \in C_A(B)$ will give a trace map of the form $\Phi'(1_A)$ where $\Phi^{-1}\Phi'(a)=a\mu$ as in Proposition~\ref{prop:Frob-data-uniqueness2}.
  }
\end{proof}

\begin{lem}
  There exists a trace map $\tr \colon \prescript{\beta}{B}A^\alpha_B \to \{\lambda,\pi\} \prescript{}{B}B_B $  if and only if there exists a $(B,B)$-bimodule homomorphism $\langle\ ,\ \rangle \colon \prescript{\beta}{B}A \times A^\alpha_B \to \{\lambda,\pi\} \prescript{}{B}B_B$ satisfying:
  \begin{enumerate}
    \item \label{lem-item:bilinear-form-invariance} $\langle xy,z\rangle = \langle x,yz\rangle$ for all $x,y,z \in A$,
    \item \label{lem-item:bilinear-form-nondegen} if $a\in A$ is such that $\langle x,a\rangle=0$ for all $x\in A$, then $a=0$,
    \item \label{lem-item:bilinear-form-surj} for every $\varphi \in \Hom_B^{\rL} (\prescript{\beta}{B}A,\prescript{}{B}B)$ there exists an $a \in A$ such that $\varphi(x) = \langle x,a\rangle $ for all $x \in A$.
  \end{enumerate}
  Consequently, if \ref{L1} is satisfied, then $A$ is an $(\alpha,\beta)$-Frobenius extension of $B$ of degree $(-\lambda,\pi)$ if and only if such a bilinear map exists.
\end{lem}

\begin{proof}
  Given a trace map $\tr$, one sets $\langle x,y\rangle =\tr(xy)$ for all $x,y \in A$.  Conversely, given such a bilinear map $\langle\ ,\ \rangle$, one sets $\tr(x)=\langle x,1_A\rangle$.  The proof that these maps have the desired properties is a straightforward verification.
  \details{
    Given a trace map $\tr$, clearly $\langle x,y\rangle =\tr(xy)$ satisfies \eqref{lem-item:bilinear-form-invariance}.  On the other hand, if $\langle\ ,\ \rangle$ satisfies \eqref{lem-item:bilinear-form-invariance}, then we have $\langle x,y\rangle =\langle xy,1_A\rangle$ so that the trace map defined by $\tr(x) = \langle x,1_A \rangle$ satisfies $\langle x,y\rangle=\tr(xy)$.  Then property \eqref{lem-item:bilinear-form-nondegen} is equivalent to \ref{L2} and property \eqref{lem-item:bilinear-form-surj} is equivalent to \ref{L3}.
  }
  The final statement follows immediately from Corollary~\ref{cor:left-trace-map}.
\end{proof}

\begin{prop} \label{prop:dual-bases}
  We have that $A$ is an $(\alpha,\beta)$-Frobenius extension of $B$ of degree $(-\lambda,\pi)$ if and only if there exists a homomorphism of $(B,B)$-bimodules $\tr \colon \prescript{\beta}{B}A^\alpha_B \to \{\lambda,\pi\} \prescript{}{B}B_B$ and finite subsets $\{x_1,\dotsc, x_n\}$, $\{y_1,\dotsc,y_n\}$ of $A$ such that $(|y_i|, \bar y_i) = (\lambda - |x_i|, \pi + \bar x_i)$ for $i=1,\dotsc,n$, and
  \begin{equation} \label{eq:dual-bases}
    a = (-1)^{\pi \bar a} \sum_{i=1}^n (-1)^{\pi \bar x_i} \beta \big( \tr(ay_i) \big) x_i = \sum_{i=1}^n \alpha^{-1}(y_i) \tr \big( x_i\alpha(a) \big) \quad \text{for all } a \in A,
  \end{equation}
  where we view $\tr$ as a degree $(-\lambda,\pi)$ map to $\prescript{}{B}B_B$ via the isomorphisms \eqref{eq:Lhom-shift-identification} and \eqref{eq:Rhom-shift-identification}.  We call the sets $\{x_1,\dotsc,x_n\}$ and $\{y_1,\dotsc,y_n\}$ \emph{dual sets of generators} of $A$ over $B$.
\end{prop}

\begin{proof}
  Suppose $A$ is an $(\alpha,\beta)$-Frobenius extension of $B$ of degree $(-\lambda,\pi)$, and let $\tr$ denote a trace map.  By Lemma~\ref{lem:proj-basis}, we can choose a left projective basis
  \[
    \{x_i \in \prescript{\beta}{B}A_{\mu_i,\gamma_i}\}_{i \in I} \quad \text{and} \quad \{\varphi_i \in \HOM_B^{\rL} (\prescript{\beta}{B}A, \prescript{}{B}B)_{-\mu_i,\gamma_i}\}_{i \in I}
  \]
  of $\prescript{\beta}{B}A$ with $I$ finite.  For $i \in I$, we can view $\varphi$ as an element of $\HOM_B^{\rL} (\prescript{\beta}{B}A, \{\lambda,\pi\} \prescript{}{B}B)_{\lambda-\mu_i,\pi+\gamma}$ via the isomorphism \eqref{eq:HOM-isom}.  Thus, by property \ref{L4}, there exists an element $y_i \in A_{\lambda-\mu_i,\pi+\gamma_i}$ such that $(-1)^{\pi \gamma_i} \varphi_i = \tr \circ \pr{y_i}$.  Then, for $a \in A$,
  \begin{equation} \label{eq:proj-basis-decomp}
    a = \sum_{i \in I} (-1)^{\bar a \gamma_i} \varphi_i(a)\cdot x_i
    = (-1)^{\pi \bar a} \sum_{i \in I} (-1)^{\pi \bar x_i} \tr(ay_i) \cdot x_i \\
    = (-1)^{\pi \bar a} \sum_{i \in I} (-1)^{\pi \bar x_i} \beta \big( \tr(ay_i) \big) x_i.
  \end{equation}

  By Proposition~\ref{prop:left-trace-map}, there exists an $(A,B)$-bimodule isomorphism
  \[
    \Phi \colon \prescript{}{A}A_B \cong \HOM^\rL_B(\prescript{\beta}{B}{A}^\alpha_A, \{\lambda, \pi\} \prescript{}{B}B_B),
  \]
  which we view as a degree $(-\lambda,\pi)$ map to $\HOM_B^{\rL}(\prescript{\beta}{B}A^\alpha_A, \prescript{}{B}B_B )$ via \eqref{eq:Lhom-shift-identification} and \eqref{eq:Rhom-shift-identification}, such that $\tr = \Phi(1_A)$.  Thus, for all $z \in \prescript{\beta}{B}A^\alpha_A$,
  \begin{multline*}
    \tr(zy_i) = \Phi(1_A)(z y_i) = (-1)^{\bar z \bar y_i} \Phi(1_A) \circ \pr{y_i}(z) \\
    = (-1)^{\bar z \bar y_i} \big( \alpha^{-1}(y_i) \cdot \Phi(1_A) \big) (z) = (-1)^{\bar z \bar y_i}  \Phi(\alpha^{-1}(y_i))(z)
  \end{multline*}
  and so, for any $\varphi \in \HOM_B^{\rL}(\prescript{\beta}{B}A^\alpha_A, \prescript{}{B}B_B )$, we have
  \begin{multline*}
    \varphi(z)
    = (-1)^{\bar z \pi} \sum_{i=1}^n (-1)^{\pi \gamma_i} \varphi \big( \beta(\tr(zy_i))x_i \big)
    = \sum_{i=1}^n (-1)^{(\bar z + \gamma_i)\pi+\bar\varphi(\bar z +\gamma_i)} \tr(zy_i) \varphi(x_i) \\
    = \sum_{i=1}^n (-1)^{\pi \gamma_i + \bar z \gamma_i+\bar \varphi(\bar z+\gamma_i)} \big( \Phi(\alpha^{-1}(y_i))(z) \big) \varphi(x_i).
  \end{multline*}
  In particular, if we take $\varphi=\Phi(a)$, we get
  \begin{multline*}
    \Phi(a)(z) = \sum_{i=1}^n (-1)^{\pi \gamma_i + \bar z \gamma_i + (\bar a + \pi)(\bar z +\gamma_i)} \big( \Phi(\alpha^{-1}(y_i))(z) \big) \big( \Phi(a)(x_i) \big) \\
    = \sum_{i=1}^n (-1)^{\pi \gamma_i + \gamma_i}\ \pr{\big( \Phi(a)(x_i) \big)} \circ \Phi \big( \alpha^{-1}(y_i) \big) (z)
    = \sum_{i=1}^n (-1)^{\gamma_i \bar a}\ \Phi \Big( \alpha^{-1}(y_i) \big( \Phi(a)(x_i) \big) \Big) (z) \\
    = \sum_{i=1}^n \Phi \Big( \alpha^{-1}(y_i) \tr(x_i \alpha(a)) \Big) (z)
  \end{multline*}
  and we conclude that $a =\sum_{i=1}^n \alpha^{-1}(y_i) \tr(x_i\alpha(a))$.

  Conversely, suppose we have a homomorphism of $(B,B)$-bimodules $\tr \colon \prescript{\beta}{B}A^\alpha_B \to \{\lambda, \pi\} \prescript{}{B}B_B$ and finite subsets $\{x_1,\dotsc, x_n\}$, $\{y_1,\dotsc,y_n\}$ of $A$ satisfying~\eqref{eq:dual-bases}.  Then, by reversing the string of equalities~\eqref{eq:proj-basis-decomp}, we see that
  \[
    \{x_i\}_{i \in I} \quad \text{and} \quad \{\varphi_i = (-1)^{\pi \gamma_i} \tr \circ \pr{y_i}\}_{i \in I}
  \]
  form a projective basis of $\prescript{\beta}{B}A$, and so property \ref{L1} holds by Lemma~\ref{lem:proj-basis}.   If $a\in A$ is such that $\tr(xa)=0$ for every $x\in A$ then
  \[
    \alpha^{-1}(a)=\sum_{i=1}^n \alpha^{-1}(y_i) \tr(x_i a) = 0,
  \]
  and so $a=0$.  Hence $\tr$ satisfies \ref{L3}.  Now suppose $\varphi \in \HOM_B^{\rL}(\prescript{\beta}{B}A, \{\lambda,\pi\} \prescript{}{B}B)$.  Then, for all $a\in A$, we have
  \begin{multline*}
    \varphi(a) = \varphi \left( (-1)^{\bar a \pi} \sum_{i=1}^n (-1)^{\pi \gamma_i} \beta \big( \tr(ay_i) \big) x_i \right)
    = \sum_{i=1}^n \tr (a y_i) \varphi(x_i) \\
    = \sum_{i=1}^n \tr \big( a y_i \alpha( \varphi(x_i) ) \big)
    = \tr \circ \pr{\left( \sum_{i=1}^n y_i \alpha( \varphi(x_i) ) \right)}(a),
  \end{multline*}
  and so $\tr$ satisfies property \ref{L4}.  Thus $A$ is an $(\alpha,\beta)$-Frobenius extension of $B$ by Corollary~\ref{cor:left-trace-map}.
\end{proof}

We note that if $A$ is an $(\alpha,\beta)$-Frobenius extension of $B$ and $\{x_1,\dotsc,x_n\}$ and $\{y_1,\dotsc,y_n\}$ are as in Proposition~\ref{prop:dual-bases}, then the elements $x_i$, along with the maps $(-1)^{\bar x_i} \beta \circ \tr \circ \pr y_i$, $i=1,\dotsc,n$, form a projective basis of $A$ as a left $B$-module, while the elements $\alpha^{-1}(y_i)$, along with the maps $\tr\circ \pl x_i\circ \alpha$, $i=1,\dotsc,n$ form a projective basis of $A$ as a right $B$-module.
\details{The index sets are finite, so we only need to verify the spanning property of projective bases.  The first assertion follows from
\[
  a = \sum_{i=1}^n (-1)^{(\bar a + \bar x_i) \pi} \beta (\tr(ay_i))x_i =\sum_{i=1}^n (-1)^{(\bar a + \pi) \bar x_i} (\beta \circ \tr \circ \pr y_i)(a) x_i,
\]
while the second follows from
\[
  a= \sum_{i=1}^n \alpha^{-1}(y_i) \tr(x_i\alpha(a)) = \sum_{i=1}^n \alpha^{-1}(y_i) (\tr \circ \pl x_i \circ \alpha)(a).
\]}

%
\section{Nakayama isomorphisms} \label{sec:Nakayama}
%

Throughout this section, we fix $A$ to be an $(\alpha,\beta)$-Frobenius extension of $B$ of degree $(-\lambda,\pi)$.  Let $\tr \colon \prescript{\beta}{B}{A}^\alpha_B \to \{\lambda,\pi\} \prescript{}{B}B_B $ be a trace map for this extension.  For any homogeneous element $c \in C_A(B)$,
we get a homomorphism of left $B$-modules $\sigma_c\colon \prescript{\beta}{B}A \to \{\lambda + |c|,\pi + \bar c\} \prescript{}{B}B $ defined by $\sigma_c = \tr \circ \pl{c}$.
\details{
  For $x \in \prescript{\beta}{B}A$, $b \in B$, and $c \in C_A(B)$, we have
  \[
    \sigma_c(b \cdot x) = \sigma_c (\beta(b)x) = \tr(c\beta(b)x) = (-1)^{\bar c \bar b} \tr(\beta(b)cx) = (-1)^{(\bar c + \pi) \bar b} b\, \tr(cx) = (-1)^{(\bar c + \pi) \bar b} b \sigma_c(x).
  \]
}

By property \ref{R3} and the linearity of the trace map, we have a grading-preserving injective linear map
\begin{align*}
  \sigma \colon C_A(B) \hookrightarrow \HOM_B^{\rL} \left( \prescript{\beta}{B}A,\{\lambda,\pi\} \prescript{}{B}B \right),\quad c \mapsto \sigma_c.
\end{align*}
Similarly, for any homogeneous $a \in A$, one obtains a homomorphism of left $B$-modules $\tau_a \colon \prescript{\beta}{B}A \to \{\lambda + |a|, \pi + \bar a\} \prescript{}{B}B$ defined by $\tau_a = \tr \circ \pr{a}$.
\details{
  For $x \in \prescript{\beta}{B}A$, $b \in B$, and $a \in A$, we have
  \[
    \tau_a(b \cdot x)
    = \tau_a(\beta(b)x)
    = (-1)^{(\bar b + \bar x) \bar a} \tr(\beta(b)x a)
    = (-1)^{(\bar b + \bar x) \bar a + \pi \bar b} b\, \tr(x a)
    = (-1)^{\bar b (\pi + \bar a)} b \tau_a(x).
  \]
}
By properties \ref{L3} and \ref{L4}, the grading-preserving linear map
\[
  \tau \colon A \to \HOM_B^{\rL} \left( \prescript{\beta}{B}A, \{\lambda,\pi\} \prescript{}{B}B \right),\quad a \mapsto \tau_a,
\]
is bijective.  Then we have an injective grading-preserving map
\begin{equation}
  \psi := \tau^{-1} \circ \sigma \colon C_A(B) \hookrightarrow A.
\end{equation}
We note that $\psi$ can be characterized as follows: for any $c \in C_A(B)$, $\psi(c)$ is the unique element of $A$ such that
\begin{equation} \label{eq:Nakayama-defining-property}
  \tr(cx) = (-1)^{\bar x \bar c} \tr(x\psi(c)) \quad \text{for all } x \in A.
\end{equation}
\details{
  For $c \in C_A(B)$ and $x \in A$, we have
  \[
    \tr(cx) = \sigma(c)(x) = \tau(\psi(c))(x) = (-1)^{\bar x \bar c}\, \tr(x \psi(c)).
  \]
  Uniqueness follows from \ref{L3}.
}

\begin{lem} \label{lem:nakayama-image}
  The image of the map $\psi$ is $C_A(\alpha(B))$.
\end{lem}

\begin{proof}
  Suppose $y \in \im \psi$ is homogeneous.  Then there is some $c\in C_A(B)$, with $\bar c = \bar y$, for which $\tr(cx) = (-1)^{\bar x \bar c} \tr(xy)$ for all $x \in A$.  But then, for any $b\in B$, we have
  \begin{equation} \label{eq:centralizer}
    \tr(x\alpha(b)y)
    = (-1)^{(\bar x + \bar b)\bar y} \tr(cx\alpha(b))
    = (-1)^{(\bar x + \bar b)\bar y} \tr(cx)b
    = (-1)^{\bar b \bar y} \tr(xy)b
    = (-1)^{\bar b \bar y} \tr(xy\alpha(b))
  \end{equation}
  for all $x\in A$, and it follows from \ref{L3} that $y\in C_A(\alpha(B))$.  Hence $\im \psi \subseteq C_A(\alpha(B))$.

  For the reverse inclusion, set $\tr^\rR=\beta\circ \tr\circ \alpha$, which is a right trace map by Lemma~\ref{lem:left-right-trace-map}.  Then, for any homogeneous $c\in C_A(B)$, the map $\sigma'_c := \tr^\rR \circ \pr{c} \colon A_B^{\beta^{-1}}\to \{\lambda + |c|, \pi + \bar c\} B_B$ is a homomorphism of right $B$-modules.  By \ref{L3}, we get a degree-preserving injection
  \[
    \sigma' \colon C_A(B) \hookrightarrow \HOM_{B}^\rR \left( A_B^{\beta^{-1}}, \{\lambda,\pi\} B_B \right), \quad c\mapsto \sigma'_c.
  \]
  Similarly, for any homogeneous $a\in A$, we have a homomorphism of right $B$-modules $\tau'_a := \tr^\rR \circ \pl{a} \colon A_B^{\beta^{-1}} \to \{\lambda + |a|, \pi + \bar a\} B_B$.  By \ref{R3} and \ref{R4}, the map
  \[
    \tau' \colon A \to \HOM_B^\rR \left( A_B^{\beta^{-1}}, \{\lambda,\pi\} B_B \right), \quad a\mapsto \tau'(a),
  \]
  is bijective.  The composition $\psi' := (\tau')^{-1} \circ \sigma' \colon C_A(B)\to A$ assigns to any $c \in C_A(B)$ the unique element $\psi'(c) \in A$ such that $\tr^\rR(xc) = (-1)^{\bar x \bar c} \tr^\rR(\psi'(c) x)$ for all $x \in A$.
  \details{
    For $c \in C_A(B)$ and $x \in A$, we have
    \[
      \tr^\rR(xc)
      = (-1)^{\bar x \bar c} \sigma'(c)(x)
      = (-1)^{\bar x \bar c} \tau'(\psi'(c))(x)
      = (-1)^{\bar x \bar c} \tr^\rR(\psi'(c) x).
    \]
  }
  Now let $y\in C_A(\alpha(B))=\alpha(C_A(B))$.  Then, since $\tr^\rR = \beta\circ \tr\circ \alpha$, we see that $\alpha \circ \psi' \circ \alpha^{-1}(y)$ is the unique element of $A$ such that $\tr(xy) = (-1)^{\bar x \bar y} \tr(\alpha \circ \psi' \circ \alpha^{-1}(y) x)$ for all $x \in A$.
  \details{
    We have
    \[
      \tr(xy)
      = \beta^{-1} \circ \tr^\rR (\alpha^{-1}(xy))
      = (-1)^{\bar x \bar y} \beta^{-1} \circ \tr^\rR( \psi'(\alpha^{-1}(y)) \alpha^{-1}(x)) \\
      = (-1)^{\bar x \bar y} \tr (\alpha \circ \psi' \circ \alpha^{-1}(y) x).
    \]
  }
  A computation similar to~\eqref{eq:centralizer} shows that $\alpha \circ \psi' \circ \alpha^{-1}(y) \in C_A(B)$, and hence $y=\psi \circ \alpha \circ \psi' \circ \alpha^{-1}(y)$.
  \details{
    Let $b\in B$, $x\in A$.  Set $a = \alpha \circ \psi' \circ \alpha^{-1}(y)$.  Note that this implies that $\bar y = \bar a$.  We have
    \[
      \tr(bax)
      = \beta^{-1}(b)\tr(ax)
      = (-1)^{\bar x \bar a} \beta^{-1}(b)\tr(xy)
      = (-1)^{\bar x \bar a} \tr(bxy)
      = (-1)^{\bar a \bar b} \tr(abx),
    \]
    hence $ba=(-1)^{\bar a \bar b} ab$ by \ref{R3}.
  }
  Thus $C_A(\alpha(B))\subseteq \im \psi$.
\end{proof}

\begin{prop}
  The map $\psi \colon C_A(B)\cong C_A(\alpha(B))$ is an isomorphism of rings.
\end{prop}

\begin{proof}
  We use the notation of the proof of Lemma~\ref{lem:nakayama-image}.  Since $\psi$ is injective, it follows from Lemma~\ref{lem:nakayama-image} that we need only show that $\psi$ is a ring homomorphism.  First, we have $\tr(1_A x)=\tr(x1_A)$ for all $x \in A$ so that $\psi(1_A)=1_A$.  Next, note that $\psi = \tau^{-1} \circ \sigma$ is grading-preserving and linear since $\sigma$ and $\tau$ are.  Finally, for $c_1,c_2 \in C_A(B)$, we have
  \[
    \tr(x\psi(c_1c_2))
    = (-1)^{\bar x (\bar c_1 + \bar c_2)} \tr(c_1c_2x)
    = (-1)^{(\bar x + \bar c_1) c_2} \tr(c_2 x\psi(c_1))
    = \tr(x\psi(c_1)\psi(c_2)),
  \]
  from which it follows, by \ref{L3}, that $\psi(c_1c_2)=\psi(c_1)\psi(c_2)$.
\end{proof}

\begin{defin}[Nakayama isomorphism]
  We call the map $\psi \colon C_A(B) \to C_A(\alpha(B))$ characterized by~\eqref{eq:Nakayama-defining-property} the \emph{Nakayama isomorphism} associated to the trace map $\tr$.
\end{defin}

\begin{rem}
  In the case that $\alpha = \id_A$ and $\beta = \id_B$, the map $\psi$ defined above is the usual Nakayama automorphism of the Frobenius extension $A$ of $B$.
\end{rem}

\begin{prop}
  Suppose $\tr_1$ and $\tr_2$ are both trace maps for the twisted Frobenius extension $A$ of $B$, and let $\psi_1, \psi_2$ be the associated Nakayama isomorphisms.  Then there exists an invertible element $a \in C_A(\alpha(B))$ of degree $(0,0)$ such that $\psi_2(c) = a \psi_1(c) a^{-1}$ for all $c \in C_A(B)$.

  Conversely, let $\tr_1$ be a trace map for this twisted Frobenius extension, with Nakayama isomorphism $\psi_1$.  If $a$ is an invertible element of $C_A(\alpha(B))$ of degree $(0,0)$,  then the map $\psi_2 \colon C_A(B)\cong C_A(\alpha(B))$ defined by $\psi_2(c) = a \psi_1(c) a^{-1}$, $c \in C_A(B)$, is the Nakayama isomorphism associated to the trace map $\tr_2 = \tr_1 \circ \pr{a}$.
\end{prop}

\begin{proof}
  Let $c\in C_A(B)$.  By Proposition~\ref{prop:trace-map-uniqueness}, there exists an invertible $a \in C_A(\alpha(B))$ of degree $(0,0)$ such that $\tr_2= \tr_1 \circ \pr{a}$.  This gives, for $x \in A$,
  \[
    \tr_2(cx)
    = \tr_1(cxa)
    = (-1)^{\bar x \bar c}\, \tr_1(xa\psi_1(c))
    = (-1)^{\bar x \bar c}\, \tr_2(x a \psi_1(c) a^{-1})
  \]
  and it follows that $\psi_2(c) = a \psi_1(c) a^{-1}$.  Conversely, defining $\psi_2$ as in the proposition will yield the Nakayama isomorphism of the map $\tr_2 := \tr_1 \circ \pr{a}$, which is guaranteed to be a trace map by Proposition~\ref{prop:trace-map-uniqueness}.
\end{proof}

\begin{prop}
  If we fix a trace map $\tr$ and choose dual sets of generators $\{x_1,\dots,x_n\}$ and $\{y_1,\dots,y_n\}$ as in Proposition~\ref{prop:dual-bases}, then the Nakayama isomorphism corresponding to $\tr$ is given explicitly by
  \[
    \psi(a) = \sum_{i=1}^n (-1)^{\bar a \bar x_i} y_i\alpha (\tr(ax_i))
  \]
  for all $a \in C_A(B)$, where we view $\tr$ as a degree $(-\lambda,\pi)$ map to $\prescript{}{B}B_B$ via the isomorphisms \eqref{eq:Lhom-shift-identification} and \eqref{eq:Rhom-shift-identification}.  Its inverse is given explicitly by
  \[
    \psi^{-1}(a)=\sum_{i=1}^n (-1)^{(\bar a + \pi) \bar x_i}\beta(\tr(y_i a))x_i
  \]
  for all $a \in C_A(\alpha(B))$.
\end{prop}

\begin{proof}
  Suppose $a \in C_A(B)$.  Replacing $a$ by $\alpha^{-1}(\psi(a))$ in the last expression in~\eqref{eq:dual-bases} yields
  \begin{gather*}
    \alpha^{-1}(\psi(a))
    = \sum_{i=1}^n \alpha^{-1}(y_i) \tr(x_i\psi(a))
    = \sum_{i=1}^n (-1)^{\bar a \bar x_i} \alpha^{-1}(y_i) \tr(ax_i) \\
    \implies \psi(a) = \sum_{i=1}^n (-1)^{\bar a \bar x_i} y_i\alpha (\tr(ax_i)).
  \end{gather*}
  Now suppose $a \in C_A(\alpha(B))$.  Replacing $a$ by $\psi^{-1}(a)$ in the second expression in~\eqref{eq:dual-bases} yields
  \[
    \psi^{-1}(a)
    = (-1)^{\bar a \pi} \sum_{i=1}^n (-1)^{\pi \bar x_i} \beta(\tr(\psi^{-1}(a)y_i))x_i
    = \sum_{i=1}^n (-1)^{(\bar a + \pi) \bar x_i} \beta \big( \tr(y_i a) \big) x_i. \qedhere
  \]
\end{proof}

%
\section{Adjointness properties} \label{sec:adjoints}
%

We now examine the relationship between twisted Frobenius extensions and induction/restriction functors.  This is one of the main motivations for the concept of twisted Frobenius extensions. Throughout this section, we assume that $A$ and $B$ are rings, $\alpha$ is an automorphism of $A$, $\beta$ is an automorphism of $B$, $\lambda \in \Lambda$, and $\pi \in \Z_2$.

\begin{lem} \label{lem:casimir-map}
  Suppose that $A$ is an $(\alpha,\beta)$-Frobenius extension of $B$ of degree $(-\lambda,\pi)$ with dual sets of generators $\{x_i\}_{i=1}^n$, $\{y_i\}_{i=1}^n$ as in Proposition~\ref{prop:dual-bases}. Then the map
  \begin{equation}
    \eta \colon \prescript{}{A}{A}_A \to \prescript{}{A}A_B \otimes_B \{-\lambda,\pi\} \twBA,\quad a \mapsto \sum_{i=1}^n a\alpha^{-1}(y_i) \otimes \{-\lambda,\pi\} x_i,\quad a \in \prescript{}{A}A_A,
  \end{equation}
  is a homomorphism of $(A,A)$-bimodules.
\end{lem}

\begin{proof}
  It suffices to show that
  \[
    a' \cdot \left(\sum_{i=1}^n \alpha^{-1}(y_i) \otimes \{-\lambda,\pi\} x_i \right) = \left(\sum_{i=1}^n \alpha^{-1}(y_i) \otimes \{-\lambda,\pi\} x_i \right) \cdot a'\quad \text{for all } a' \in A.
  \]
  We have
  \begin{multline*}
    a' \cdot \left(\sum_{i=1}^n \alpha^{-1}(y_i) \otimes \{-\lambda,\pi\} x_i \right)
    = \sum_{i=1}^n a' \alpha^{-1}(y_i) \otimes \{-\lambda,\pi\} x_i \\
    = \sum_{i=1}^n \left( \sum_{j=1}^n \alpha^{-1}(y_j) \tr \big( x_j \alpha(a') y_i \big) \right) \otimes \{-\lambda,\pi\} x_i \\
    = \sum_{i,j=1}^n (-1)^{\pi (\bar a' + \bar x_i + \bar x_j)} \alpha^{-1}(y_j) \otimes \{-\lambda,\pi\} \Big( \beta \big( \tr(x_j \alpha(a') y_i) \big) x_i \Big) \\
    = \sum_{j=1}^n \alpha^{-1}(y_j) \otimes\{-\lambda,\pi\} \left( (-1)^{\pi (\bar x_j + \bar a')} \sum_{i=1}^n (-1)^{\pi \bar x_i} \beta \big( \tr(x_j \alpha(a') y_i) \big) x_i \right) \\
    = \sum_{j=1}^n \alpha^{-1}(y_j) \otimes \{-\lambda,\pi\} x_j \alpha(a')
    = \left( \sum_{j=1}^n \alpha^{-1}(y_j) \otimes \{-\lambda,\pi\} x_j \right) \cdot a',
  \end{multline*}
  where, in the second equality, we used the last expression in~\eqref{eq:dual-bases} with $a = a' \alpha^{-1}(y_i)$ and, in the fifth equality, we used the second expression in~\eqref{eq:dual-bases} with $a=x_j \alpha(a')$.
  \details{
    It now follows that $\eta$ is a homomorphism of $(A,A)$-bimodules.  The fact that it commutes with the left $A$-action follows immediately from the definition.  What we have shown above is that $a' \cdot \eta(1_A) = \eta(1_A) \cdot a'$ for all $a' \in A$.  Now suppose $a \in \prescript{}{A}A_A$ and $a' \in A$.  Then
    \begin{gather*}
      \eta(a) \cdot a' = (a \cdot \eta(1_A)) \cdot a' = (\eta(1_A) \cdot a) \cdot a' = \eta(1_A) \cdot (aa') = (aa') \cdot \eta(1_A) = \eta(aa'),
    \end{gather*}
    and so $\eta$ also commutes with the right $A$-action.
  }
\end{proof}

Define
\begin{equation}
  \varepsilon \colon \{-\lambda,\pi\} \twBA \otimes_A \prescript{}{A}{A}_B \to \prescript{}{B}{B}_B,\quad \{-\lambda,\pi\} a_1 \otimes a_2 \mapsto \tr(a_1 \alpha(a_2)),
\end{equation}
where we view $\tr$ as a degree $(-\lambda,\pi)$ map to $\prescript{}{B}B_B$ via \eqref{eq:Lhom-shift-identification} and \eqref{eq:Rhom-shift-identification}.  Then it is straightforward to verify that $\varepsilon$ is a homomorphism of $(B,B)$-bimodules.
\details{
  For $a_1 \in \twBA$, $a_2 \in \prescript{}{A}A_B$ and $b \in B$, we have
  \begin{multline*}
    \varepsilon(b \cdot (\{-\lambda,\pi\}a_1 \otimes a_2)) = \varepsilon \left( (-1)^{\pi \bar b} \{-\lambda,\pi\} (\beta(b)a_1) \otimes a_2 \right) = (-1)^{\pi \bar b} \tr \big( \beta(b)a_1 \alpha(a_2) \big) \\
    = b \tr(a_1\alpha(a_2))
    = b \varepsilon(\{-\lambda,\pi\}a_1 \otimes a_2)
  \end{multline*}
  and
  \[
    \varepsilon ((\{-\lambda,\pi\}a_1 \otimes a_2) \cdot b)
    = \varepsilon (\{-\lambda,\pi\} a_1 \otimes a_2b)
    = \tr(a_1 \alpha(a_2) \alpha(b))
    = \tr(a_1 \alpha(a_2)) b
    = \varepsilon(\{-\lambda,\pi\} a_1 \otimes a_2) b.
  \]
}

Recall that, for a ring $R$, we let $R\md$ denote the category of left $R$-modules.  If $R$ and $S$ are rings and $M$ is an $(R,S)$-bimodule, then we have a functor
\begin{gather*}
  M \otimes_S - \colon S\md \to R\md.
\end{gather*}
In particular, we have the functors
\begin{gather*}
  \prescript{}{A}A_B \otimes_B - \colon B\md \to A\md\quad \text{and} \\
  \prescript{}{B}A_A \otimes_A - \colon A\md \to B\md
\end{gather*}
of \emph{induction} and \emph{restriction}, respectively.  Recall that induction is left adjoint to restriction.  The following theorem can be interpreted as saying that $A$ is a twisted Frobenius extension of $B$ if and only if induction is also \emph{twisted shifted} right adjoint to restriction. In the case that all gradings are trivial, this adjunction was studied by Morita (see \cite[Th.~4.1]{Mor65}).

\begin{theo} \label{theo:twisted-adjointness}
 The functor $\prescript{}{A}{A}_B \otimes_B -$ is right adjoint to the functor $\{-\lambda,\pi\} \twBA \otimes_A -$ if and only if $A$ is an $(\alpha,\beta)$-Frobenius extension of $B$ of degree $(-\lambda,\pi)$.
\end{theo}

\begin{proof}
  Let $A$ be an $(\alpha,\beta)$-Frobenius extension of $B$ of degree $(-\lambda,\pi)$, and let $\eta$ and $\varepsilon$ be the maps defined above.  We claim that $\eta$ and $\varepsilon$ are the unit and counit of an adjunction, making $\prescript{}{A}{A}_B \otimes_B -$ right adjoint to $\{-\lambda,\pi\} \twBA \otimes_A -$.  First, we need to show that the composition
  \begin{multline*}
    \{-\lambda,\pi\} \twBA \cong \{-\lambda,\pi\} \twBA \otimes_A \prescript{}{A}{A}_A
    \xrightarrow{\id \otimes \eta} \{-\lambda,\pi\} \twBA \otimes_A \prescript{}{A}{A}_B \otimes_B \{-\lambda,\pi\} \twBA \\
    \xrightarrow{\varepsilon \otimes \id} \prescript{}{B}{B}_B \otimes_B \{-\lambda,\pi\} \twBA \cong \{-\lambda,\pi\} \twBA
  \end{multline*}
  is the identity map, where the inverse of the first isomorphism is given by the right $A$-action and the last isomorphism is given by the left $B$-action.  For $a \in \twBA$, we have
  \begin{multline*}
    \{-\lambda,\pi\} a \mapsto \{-\lambda,\pi\} a \otimes 1_A
    \mapsto \sum_{i=1}^n \{-\lambda,\pi\} a \otimes \alpha^{-1}(y_i) \otimes \{-\lambda,\pi\} x_i \\
    \mapsto \sum_{i=1}^n \tr(a y_i) \otimes \{-\lambda,\pi\} x_i
    \mapsto \{-\lambda,\pi\} (-1)^{\pi \bar a} \sum_{i=1}^n (-1)^{\pi \bar x_i}  \beta( \tr (ay_i)) x_i
    = \{-\lambda,\pi\} a,
  \end{multline*}
  as desired, where the last equality follows from \eqref{eq:dual-bases}.

  Finally, we need to show that the composition
  \[
    \prescript{}{A}{A}_B \cong \prescript{}{A}{A}_A \otimes_A \prescript{}{A}{A}_B
    \xrightarrow{\eta \otimes \id} \prescript{}{A}{A}_B \otimes_B \{-\lambda,\pi\} \twBA \otimes_A \prescript{}{A}{A}_B \xrightarrow{\id \otimes \varepsilon} \prescript{}{A}{A}_B \otimes_B \prescript{}{B}{B}_B \cong \prescript{}{A}{A}_B
  \]
  is also the identity map, where the first and last isomorphisms again come from the corresponding actions.  We have
  \begin{multline*}
    a \mapsto 1_A \otimes a
    \mapsto \sum_{i=1}^n \alpha^{-1}(y_i) \otimes \{-\lambda,\pi\} x_i \otimes a \\
    \mapsto \sum_{i=1}^n \alpha^{-1}(y_i) \otimes \tr(x_i\alpha(a))
    \mapsto \sum_{i=1}^n \alpha^{-1}(y_i) \tr(x_i \alpha(a))
    = a,
  \end{multline*}
  where the last equality again follows from \eqref{eq:dual-bases}.

  Now suppose that the functor $\prescript{}{A}{A}_B \otimes_B -$ is right adjoint to the functor $\{-\lambda,\pi\} \twBA \otimes_A -$.  Then there exist bimodule homomorphisms
  \[
    \eta \colon \prescript{}{A}{A}_A \to \prescript{}{A}A_B \otimes_B \{-\lambda,\pi\} \twBA
  \]
  and
  \[
    \varepsilon \colon \{-\lambda,\pi\} \twBA \otimes_A \prescript{}{A}{A}_B \to \prescript{}{B}{B}_B
  \]
  satisfying the counit-unit equations.  Define a map
  \[
    \tr \colon \prescript{\beta}{B} A_B^{\alpha} \to \{\lambda,\pi\} \prescript{}{B}B_B, \quad a \mapsto \{\lambda,\pi\}\varepsilon(\{-\lambda,\pi\} a \otimes 1_A).
  \]
  Then $\tr$ is a homomorphism of bimodules since $\varepsilon$ is.  Write $\eta(1_A)=\sum_{i=1}^n y_i' \otimes \{-\lambda,\pi\} x_i$ for $y_i' \in \prescript{}{A}A_B$, $x_i \in \twBA$, $i=1,\dotsc,n$.  We claim that the sets
  \[
    \{x_i \ |\ i=1,\dotsc,n\} \quad \text{and} \quad \{y_i := \alpha(y_i')\ |\ i=1,\dotsc,n\}
  \]
  form dual sets of generators as in Proposition~\ref{prop:dual-bases}.  Indeed, first note that, since $\eta$ is a degree-preserving map, we have $\bar x_i = \bar x_i' = \bar y_i + \pi = \bar y_i' + \pi$.  Then consider the composition
  \begin{multline*}
    \{-\lambda,\pi\} \twBA \cong \{-\lambda,\pi\} \twBA \otimes_A \prescript{}{A}{A}_A
    \xrightarrow{\id \otimes \eta} \{-\lambda,\pi\} \twBA \otimes_A \prescript{}{A}{A}_B \otimes_B \{-\lambda,\pi\} \twBA \\
    \xrightarrow{\varepsilon \otimes \id} \prescript{}{B}{B}_B \otimes_B \{-\lambda,\pi\} \twBA \cong \{-\lambda,\pi\} \twBA,
  \end{multline*}
  which is the identity map.  For $a \in A$, we have
  \begin{multline*}
    \{-\lambda,\pi\} a \mapsto \{-\lambda,\pi\} a \otimes 1_A
    \mapsto \sum_{i=1}^n \{-\lambda,\pi\} a \otimes y_i' \otimes \{-\lambda,\pi\} x_i \\
    = \sum_{i=1}^n \{-\lambda,\pi\} a y_i \otimes 1_A\otimes \{-\lambda,\pi\} x_i
    \mapsto \sum_{i=1}^n \tr(a y_i) \otimes \{-\lambda,\pi\} x_i \\
    \mapsto \sum_{i=1}^n (-1)^{\pi(\bar a + \bar x_i)} \{-\lambda,\pi\} \beta( \tr (a y_i)) x_i,
  \end{multline*}
  and hence $a = (-1)^{\pi \bar a} \sum_{i=1}^n (-1)^{\pi \bar x_i} \beta( \tr (a y_i)) x_i$.

  On the other hand, consider the composition
  \[
    \prescript{}{A}{A}_B \cong \prescript{}{A}{A}_A \otimes_A \prescript{}{A}{A}_B
    \xrightarrow{\eta \otimes \id} \prescript{}{A}{A}_B \otimes_B \{-\lambda,\pi\} \twBA \otimes_A \prescript{}{A}{A}_B \xrightarrow{\id \otimes \varepsilon} \prescript{}{A}{A}_B \otimes_B \prescript{}{B}{B}_B \cong \prescript{}{A}{A}_B,
  \]
  which is also the identity map.  We have
  \begin{multline*}
    a \mapsto 1_A \otimes a
    \mapsto \sum_{i=1}^n y_i' \otimes \{-\lambda,\pi\} x_i \otimes a
    =\sum_{i=1}^n y_i' \otimes \{-\lambda,\pi\} x_i \alpha(a) \otimes 1_A \\
    \mapsto \sum_{i=1}^n y_i' \otimes \tr(x_i \alpha(a))
    \mapsto \sum_{i=1}^n y_i' \tr(x_i \alpha(a))
    = \sum_{i=1}^n \alpha^{-1}(y_i) \tr(x_i \alpha(a)),
  \end{multline*}
  and so $a = \sum_{i=1}^n \alpha^{-1}(y_i) \tr(x_i \alpha(a))$.  The theorem then follows from Proposition~\ref{prop:dual-bases}.
\end{proof}

%
\section{Examples} \label{sec:examples}
%

Suppose that $\F$ is a field.  If $A_1$ and $A_2$ are graded superalgebras over $\F$, and $M$ is an $(A_1,A_2)$-bimodule, then the dual
\[
  M^\vee := \HOM_\F(M,\F)
\]
is an $(A_2,A_1)$-bimodule with action given, for $a_1 \in A_1$, $a_2 \in A_2$, $f \in M^\vee$, by
\begin{equation} \label{eq:BA-hom-bimod}
  a_2 \cdot f \cdot a_1 = (-1)^{\bar a_2 \bar f} f \circ \pr{a_2} \circ \pl{a_1}=(-1)^{\bar a_2 \bar f + \bar a_2 \bar a_1}f \circ \pl{a_1} \circ \pr{a_2} .
\end{equation}

Recall that a finite-dimensional graded superalgebra $A$ is a Frobenius graded superalgebra of degree $(-\lambda_A,\pi_A)$ if there exists a grading-preserving linear map $\tr_A \colon A \to \{\lambda_A,\pi_A\}\F$, called the trace map of $A$, such that the kernel of $\tr_A$ contains no nonzero left ideals of $A$.  Throughout this section, we will view such a trace map as a degree $(-\lambda_A, \pi_A)$ map to $\F$ via \eqref{eq:Lhom-shift-identification} and \eqref{eq:Rhom-shift-identification}.  The corresponding Nakayama automorphism is the algebra automorphism $\psi_A$ of $A$ satisfying $\tr_A(a_1 a_2) = (-1)^{\bar a_1 \bar a_2} \tr_A(a_2 \psi_A(a_1))$ for all $a_1,a_2$.  In the language of the current paper, a Frobenius graded superalgebra $A$ is an $(\id_A,\id_\F)$-Frobenius extension of $\F$.

Suppose that $A$ is a Frobenius graded superalgebra of degree $(-\lambda_A,\pi_A)$ with trace map $\tr_A$ and corresponding Nakayama automorphism $\psi_A$.  Furthermore, suppose that $B$ is a graded subalgebra of $A$ that is itself a Frobenius graded superalgebra of degree $(-\lambda_B,\pi_B)$, with trace map $\tr_B$ and corresponding Nakayama automorphism $\psi_B$.

Fix an $\F$-basis $\cB$ of $B$ and let $\cB^\vee = \{b^\vee\ |\ b \in \cB\}$ be the right dual basis defined by
\[
  \langle b_1,b_2^\vee \rangle := \tr_B(b_1b_2^\vee) = \delta_{b_1,b_2},\quad b_1,b_2 \in \cB.
\]
It follows that
\begin{equation} \label{eq:basis-decomp}
  \sum_{b \in \cB} \langle b', b^\vee \rangle b = b' = \sum_{b \in \cB} \langle b, b' \rangle b^\vee \quad \text{for all } b' \in B.
\end{equation}
\details{
  Let $b' \in B$.  Then, for all $b_1 \in \cB$, we have
  \[
    \left\langle \sum_{b \in \cB} \langle b', b^\vee \rangle b, b_1^\vee \right\rangle = \langle b', b_1^\vee \rangle\quad \text{and} \quad \left\langle b_1, \sum_{b \in \cB} \langle b,b' \rangle b^\vee \right\rangle = \langle b_1,b' \rangle,
  \]
  and the result follows from the nondegeneracy of $\langle\ ,\ \rangle$.
}
\details{
  Let $\cB^* = \{b^*\ |\ b \in \cB\}$ be the left dual basis defined by
  \[
    \langle b_1^*,b_2 \rangle := \tr_B(b_1^* b_2) = \delta_{b_1,b_2},\quad b_1,b_2 \in \cB.
  \]
  It follows that $b^* = (-1)^{\bar b(\bar b + \sigma)} \psi_B^{-1}(b^\vee)$ for $b \in \cB$.
  For $b_1, b_2 \in \cB$, we have
  \[
    (-1)^{(\bar b_1 + \sigma) \bar b_2} \tr_B(\psi_B^{-1}(b_1^\vee) b_2) = \tr_B(b_2 b_1^\vee)  = \delta_{b_1, b_2}.
  \]
}

\begin{lem}
  For any $(B,A)$-bimodule $\prescript{}{B}{M}_A$, the maps
  \begin{gather}
    \HOM^\rL_B \left( \prescript{\psi_B}{B}{M}_A, \prescript{}{B}{B}_B \right) \to \{\lambda_B, \pi_B\} (\prescript{}{B}{M}_A)^\vee,\quad f \mapsto \tr_B \circ f, \label{eq:dual-isom} \\
    \{\lambda_B, \pi_B\} (\prescript{}{B}{M}_A)^\vee \to \HOM^\rL_B \left( \prescript{\psi_B}{B}{M}_A, \prescript{}{B}{B}_B \right),\quad \theta \mapsto \left( \ m \mapsto \sum_{b \in \cB} \theta(b \cdot m) b^\vee \ \right), \label{eq:dual-isom-inverse}
  \end{gather}
  are mutually inverse isomorphisms of $(A,B)$-bimodules.
\end{lem}

\begin{proof}
  The map~\eqref{eq:dual-isom} is clearly $\F$-linear.  For $a \in A$, $b \in B$, $y \in  \prescript{\psi_B}{B}{M}_A$, and $f \in \HOM^\rL_B (\prescript{\psi_B}{B}{M}_A, \prescript{}{B}{B}_B)$, we have
  \[
    \tr_B \circ (a \cdot f) = (-1)^{\bar a \bar f}\, \tr_B \circ f \circ \pr{a} = (-1)^{\bar a \pi_B} a \cdot (\tr_B \circ f), \quad \text{and}
  \]
  \begin{multline*}
    \tr_B \circ (f \cdot b) (y)
    = (-1)^{\bar b \bar f} \tr_B \circ \pr{b} \circ f (y)
    = (-1)^{\bar b \bar y} \tr_B (f(y)b)
    = (-1)^{\bar b \bar f} \tr_B (\psi_B^{-1}(b)f(y)) \\
    = \tr_B (f(by))
    = \tr_B \circ f \circ \pl{b}(y)
    = ((\tr_B \circ f) \cdot b)(y).
  \end{multline*}
  Thus~\eqref{eq:dual-isom} is a homomorphism of $(A,B)$-bimodules.

  It remains to show that the given maps are mutually inverse.    The map~\eqref{eq:dual-isom} followed by the map~\eqref{eq:dual-isom-inverse} sends $f \in \HOM^\rL_B(\prescript{\psi_B}{B}{M}_A, \prescript{}{B}{B}_B)$ to the map
  \[
    m \mapsto \sum_{b \in \cB} \tr_B (f(b \cdot m)) b^\vee = \sum_{b \in \cB} \tr_B (bf(m)) b^\vee = \sum_{b \in \cB} \langle b, f(m) \rangle b^\vee = f(m).
  \]
  So the map~\eqref{eq:dual-isom} is a right inverse to~\eqref{eq:dual-isom-inverse}.

  On the other hand, the map~\eqref{eq:dual-isom-inverse} followed by~\eqref{eq:dual-isom} sends $\theta \in (\prescript{}{B}{M}_A)^\vee$ to the map
  \[
    m \mapsto \sum_{b \in \cB} \tr_B(\theta(b \cdot m)b^\vee)
    = \sum_{b \in \cB} \theta(b \cdot m) \tr_B(b^\vee) \\
    = \sum_{b \in \cB} \theta(\tr_B(b^\vee) b \cdot m)
    = \theta(m),
  \]
  where the last equality follows from \eqref{eq:basis-decomp}.  Thus \eqref{eq:dual-isom} is also left inverse to \eqref{eq:dual-isom-inverse}.
\end{proof}

Let
\[
  \kappa \colon \Hom^\rL_B \left( \prescript{\psi_B}{B}{A}_A^{\psi_A}, \prescript{}{B}{B}_B \right) \to \{\lambda_B, \pi_B\} \left(\prescript{}{B}{A}_A^{\psi_A} \right)^\vee,\quad \alpha \mapsto \tr_B \circ \alpha,
\]
be the special case of the isomorphism~\eqref{eq:dual-isom} of $(A,B)$-bimodules where $M = \prescript{}{B}{A}_A^{\psi_A}$.

\begin{lem}
  The map
  \[
    \varphi_A \colon \prescript{}{A}{A}_B \to \{\lambda_A, \pi_A\} \left( \prescript{}{B}{A}_A^{\psi_A} \right)^\vee ,\quad \varphi_A(a) = \tr_A \circ \pr{\psi_A(a)},
  \]
  is an isomorphism of $(A,B)$-bimodules.
\end{lem}

\begin{proof}
  The map $\varphi_A$ is clearly $\F$-linear.  Now, for $a\in A$, $b \in B$, $a' \in \prescript{}{A}A_B$, and $x \in \prescript{}{B}A_A^{\psi_A}$, we have
  \begin{multline*}
    \varphi_A (a \cdot a' \cdot b) (x)
    = \tr_A \circ \pr{\psi_A(a a' b)} (x)
    = (-1)^{\bar x( \bar a + \bar a' + \bar b)} \tr_A (x\psi_A(aa'b)) \\
    = (-1)^{\bar x (\bar a + \bar a') + \bar b (\bar a + \bar a')} \tr_A (bx \psi_A(a) \psi_A(a'))
    = (-1)^{\bar a (\bar b + \bar a')} \tr_A \circ \pr{\psi_A(a')} \circ \pl{b} \circ \pr{\psi_A(a)} (x) \\
    = (-1)^{\bar a (\bar b + \bar a')} \varphi_A(a') \circ \pl{b} \circ \pr{\psi_A(a)} (x)
    = (a \cdot \varphi_A(a') \cdot b) (x).
  \end{multline*}
  and so $\varphi_A$ is a homomorphism of $(A,B)$-bimodules.  It is injective since $A$ is a Frobenius algebra and hence $\tr_A$ has no nonzero left ideals.  Since the $\F$-dimension of the domain and codomain of $\varphi_A$ are both $\dim_\F A$, it is also surjective.
\end{proof}

\begin{prop} \label{prop:Frobenius-subrings-L2}
  The map
  \[
    \Phi \colon \prescript{}{A}A_B \xrightarrow{\cong}\Hom^\rL_B \left( \prescript{\psi_B}{B}{A}^{\psi_A}_A, \{\lambda_A-\lambda_B, \pi_A + \pi_B\} \prescript{}{B}B_B \right),\quad \Phi = \kappa^{-1} \circ \varphi_A,
  \]
  is an isomorphism of $(A,B)$-bimodules.  The corresponding map
  \[
    \tr \colon \prescript{\psi_B}{B}A^{\psi_A}_B \to \{\lambda_A-\lambda_B, \pi_A + \pi_B\} \prescript{}{B}B_B,\quad \tr(a) = \sum_{b \in \cB} \tr_A(\psi_B(b)a)b^\vee,
  \]
  is a left trace map, i.e., it is a homomorphism of $(B,B)$-bimodules satisfying conditions \ref{L3} and \ref{L4}.
\end{prop}

\begin{proof}
  The first assertion follows immediately from the fact that $\kappa$ and $\varphi_A$ are isomorphisms of $(A,B)$-bimodules.  For the second assertion, note that, for $a \in \prescript{\psi_B}{B}A^{\psi_A}_B$, we have
  \[
    \tr(a) = \Phi(1_A)(a) = (\kappa^{-1} \circ \varphi_A)(1_A)(a) = \kappa^{-1}(\tr_A)(a) = \sum_{b \in \cB} \tr_A(b \cdot a)b^\vee = \sum_{b \in \cB} \tr_A(\psi_B(b)a)b^\vee.
  \]
  That $\tr$ is a homomorphism of $(B,B)$-bimodules satisfying conditions \ref{L3} and \ref{L4} then follows from Proposition~\ref{prop:left-trace-map}.
\end{proof}

\begin{cor} \label{cor:Frobalg-implies-Frobext}
  If $A$ and $B$ are as stated at the beginning of this section, and $A$ is projective as a left $B$-module, then $A$ is a $(\psi_A,\psi_B)$-Frobenius extension of $B$ of degree $(\lambda_B-\lambda_A,\pi_B + \pi_A)$.  In particular, the induction functor $\prescript{}{A}{A}_B \otimes_B -$ is right adjoint to the twisted restriction functor $\{\lambda_B-\lambda_A,\pi_B + \pi_A\} \prescript{\psi_B}{B}{A}^{\psi_A}_A \otimes_A -$.
\end{cor}

\begin{proof}
  Since $A$ is finite dimensional, it is finitely generated as a left $B$-module.  Thus, the assumption that $A$ is also projective as a left $B$-module implies that condition \ref{L1} is satisfied.  Then Proposition~\ref{prop:Frobenius-subrings-L2} implies that condition \ref{L2} is also satisfied.  The final assertion follows immediately from Theorem~\ref{theo:twisted-adjointness}.
\end{proof}

\begin{rem} \label{rem:non-proj-eg}
  Note that a Frobenius algebra is not always projective over a Frobenius subalgebra.  For instance, if $A = \C[x]/x^3\C[x]$ and $B$ is the unital subalgebra generated by $x + x^3 \C[x]$, then it is easily seen that $A$ and $B$ are both Frobenius algebras, but that $A$ is not projective as a $B$-module.
\end{rem}

We conclude with an example of a twisted Frobenius extension that is not a Frobenius extension, nor even a Frobenius extension of the 2nd kind in the sense of \cite{NT60}.

\begin{eg}[Nilcoxeter algebras]\label{eg:Nilcoxeter}
  Fix a nonnegative integer $n$.  The \emph{nilcoxeter algebra} $N_n$ is the unital $\F$-algebra generated by $u_1,\dotsc,u_{n-1}$, subject to the relations
  \begin{gather*}
    u_i^2=0 \text{ for } i=1,2,\dotsc,n-1, \\
    u_i u_j=u_j u_i \text{ for } i,j=1,\dots,n-1 \text{ such that } |i-j|>1, \\
    u_i u_{i+1} u_i=u_{i+1} u_i u_{i+1} \text{ for } i=1,2,\dotsc,n-2.
  \end{gather*}
  By convention, we set $N_0=N_1=\F$ and consider each generator $u_i$ as being of degree $(1,1)$.

  Note that any element in the ideal generated by $u_1,\dotsc, u_{n-1}$ is a zero divisor, and hence any homogeneous invertible element in $N_n$ must be a multiple of the unit element.  For any $n$, the nilcoxeter algebra $N_n$ is a Frobenius graded superalgebra over $\F$ of degree $\left( -{n \choose 2}, {n \choose 2} \right)$, with Nakayama automorphism given on the generators by $\psi_n(u_i)=u_{n-i}$ (see, for example, \cite[Lemma 8.2]{RS15}).  In \cite[Prop.~4]{Kho01} it is shown that $N_{n}$ is free as a left $N_{n-1}$-module, and it follows by induction that $N_n$ is free as a left $N_m$-module for any $m<n$.  Thus, by Corollary \ref{cor:Frobalg-implies-Frobext}, if $m<n$, then the canonical injection $N_m \hookrightarrow N_n$ gives a $(\psi_n,\psi_m)$-Frobenius extension of degree $\left( {m \choose 2} - {n \choose 2} , {m \choose 2} + {n \choose 2} \right)$.  We claim that this extension is not an $(\id_A,\beta)$-Frobenius for any automorphism $\beta$ of $N_m$.  Indeed, suppose such a $\beta$ existed.  Then by Proposition \ref{prop:Frob-data-uniqueness}, for every $b\in B$, we have
  \[
    \psi_n(b)=\beta^{-1} \circ \psi_m(b)
  \]
  since conjugation by homogeneous invertible elements in $N_n$ is trivial.  If we take $b=u_1$, this gives
  \[
    u_{n-1}=\beta^{-1}(u_{m-1}),
  \]
  which is a contradiction since $u_{n-1}\notin N_m$.
\end{eg}



\newcommand{\arxiv}[1]{\href{http://arxiv.org/abs/#1}{\tt
  arXiv:\nolinkurl{#1}}}

\end{document}